\documentclass[reqno,10pt]{article}
\usepackage{fullpage}
\usepackage[T1]{fontenc}                                   
\usepackage{amsmath,amsthm,amssymb,amsfonts,caption,xcolor}
\usepackage{accents}

\numberwithin{equation}{section}
\theoremstyle{plain}
\newtheorem{teo}{Theorem}[section]
\newtheorem{lemma}[teo]{Lemma}
\newtheorem{cor}[teo]{Corollary}

\newtheorem{prop}[teo]{Proposition}
\newtheorem{deff}[teo]{Definition}
\theoremstyle{remark}
\newtheorem{oss}[teo]{Remark}

\def\K{\mathcal K}

\def\P{\mathcal P}

\def\N{\mathbb N}
\def\R{\mathbb R}

\def\01{[0,1]}

\def\bal{\begin{aligned}}
\def\eal{\end{aligned}}
\def\ds{\displaystyle}
\usepackage{graphicx}

\def\KL {\operatorname{KL}}

\def\gammaep {\gamma^{\ep}}

\newcommand{\be}{\begin{equation}}
\newcommand{\ee}{\end{equation}}

\usepackage{authblk}
\usepackage{xcolor,eucal,enumerate,mathrsfs}
\usepackage[normalem]{ulem}
\usepackage{amsmath,amssymb,epsfig,amsthm} 
\usepackage[latin1]{inputenc}
\usepackage{psfrag}          

\newcommand{ \ep }{\varepsilon}

\newcommand{ \refmx }{\mathfrak{m}_1}
\newcommand{ \refmy }{\mathfrak{m}_2}
\newcommand{ \refm }{\mathfrak{m}}

\newcommand{ \Lexp}{L^{\rm exp}_{\ep}}
\newcommand{ \Sep}{\mathcal{S}_{\ep}}

\newcommand{\KLN}{\operatorname{KL}^{N}}
\newcommand{ \Dep}{D_{\ep}}

\newcommand{ \OTep}{\operatorname{OT}_{\ep}}
\newcommand{ \OTNep}{\operatorname{OT}^N_{\ep}}
\newcommand{ \ka}{\mathcal{K}}

\newcommand{\restr}[1]{\lower3pt\hbox{$|_{#1}$}}

\newcommand{\Fcep}{\mathcal{F}^{(c, \ep)}}
\newcommand{\uunocep}{\hat{u}^{(N,c,\ep)}_1}
\newcommand{\uduecep}{\hat{u}^{(N,c,\ep)}_2}
\newcommand{\Ficep}{\mathcal{F}_i^{(N,c, \ep)}}

\newcommand{\uicep}{\hat{u}^{(N,c,\ep)}_i}

\newcommand{\uNcep}{\hat{u}^{(N,c,\ep)}_N}
\newcommand{\ucep}{u^{(c,\ep)}}
\newcommand{\vcep}{v^{(c,\ep)}}

\newcommand{\rhoN}{\rho^N}
\newcommand{\Fam}{\mathcal{F}}
\newcommand{\ac}{\ll}

\usepackage{color}
\usepackage{pstricks}
\usepackage{ifthen}
\newrgbcolor{simocolor}{.7 .1 .1}
\newrgbcolor{augucolor}{.1 .1 .7}
\usepackage{csquotes}

\newcommand\blfootnote[1]{%
  \begingroup
  \renewcommand\thefootnote{}\footnote{#1}%
  \addtocounter{footnote}{-1}%
  \endgroup
}

\newboolean{DEBUG}
\setboolean{DEBUG}{true}

\ifthenelse {\boolean{DEBUG}}
{}

\ifthenelse {\boolean{DEBUG}}
{}

\title{An Optimal Transport approach for the Schr\"odinger bridge problem and convergence of Sinkhorn algorithm}

\author[1,2]{Simone Di Marino}
\author[3]{Augusto Gerolin}
\affil[1]{\small{INdAM, Unit\`a SNS, Scuola Normale Superiore, Piazza dei Cavalieri 7, 56126 Pisa (PI), Italy.}}
\affil[2]{Universit\`a degli Studi di Genova, DIMA, Genova, Italy.}
\affil[3]{Vrije Universiteit Amsterdam, Department of Theoretical Chemistry, Amsterdam, The Netherlands.}

\date{\today}

\begin{document}
\maketitle

\blfootnote{simone.dimarino@unige.it, augustogerolin@gmail.com}
\blfootnote{S. DM. is member of \textit{Gruppo Nazionale per l'Analisi Matematica, la Probabilit\`a e le loro Applicazioni} (GNAMPA) of the \textit{Istituto Nazionale di Alta Matematica} (INdAM).  A.G. acknowledges the European Research Council under H2020/MSCA-IF\textit{``OTmeetsDFT"} [grant ID: 795942]. }
\begin{abstract}
\noindent
This paper exploit the equivalence between the Schr\"odinger Bridge problem \cite{LeoSurvey,NelNM,Schr31} and the entropy penalized optimal transport \cite{Cut,GalSal10} in order to find a different approach to the duality, in the spirit of optimal transport. This approach results in a priori estimates which are consistent in the limit when the regularization parameter goes to zero. In particular, we find a new proof of the existence of maximising entropic-potentials and therefore, the existence of a solution of the Schr\"odinger system. Our method extends also when we have more than two marginals:  we can provide an alternative proof of the convergence of the Sinkhorn algorithm with two margianals and we show that the Sinkhorn algorithm converges in the multi-marginal case. \end{abstract}

\noindent
{\small Keywords: Schr\"odinger problem, Entropic regularization of Optimal Transport, Kantorovich duality, Sinkhorn algorithm, Iterative Proportional Fitting Procedure.}

\section{Introduction}
\quad Let $(X,d_X)$ and $(Y,d_Y)$ be Polish spaces, $c:X\times Y\to \mathbb{R}$ be a cost function, $\rho_1\in\P(X)$ and $\rho_2\in \P(Y)$ be probability measures. We consider the \textit{Schr\"odinger Bridge problem}

\begin{equation}\label{intro:mainKL}
\OTep(\rho_1,\rho_2) = \inf_{\gamma\in \Pi(\rho_1,\rho_2)}\ep\KL(\gamma|\ka),
\end{equation}
where $\mathcal{K}$ is the so-called \textit{Gibbs Kernel} associated with the cost $c$: 
 \begin{equation}\label{eqn:Gibbs}
 \mathcal{K} = k(x,y)\rho_1\otimes\rho_2 = e^{-\frac{c(x,y)}{\ep}}\rho_1\otimes\rho_2.
 \end{equation}
 The function $\KL(\gamma|\ka)$ in \eqref{intro:mainKL} is the Kullback-Leibler divergence between the probability measures $\gamma$ and $\ka \in \P(X\times Y)$ which is defined as
\[
\ds\KL(\gamma|\ka) =
     \ds\begin{cases}
      \ds\int_{X\times Y}\gamma\log\left(\dfrac{\gamma}{k}\right)d(\rho_1\otimes\rho_2) &\text{ if } \gamma \ll \rho_1\otimes\rho_2 \\
       +\infty &\text{ otherwise }\\ 
     \end{cases}.
\]
Here, by abuse of notation we are denoting by $\gamma$ the Radon-Nikodym derivative of $\gamma$ with respect to the product measure $\rho_1\otimes\rho_2$. Geometrically speaking, when we interpret the Kullback-Leibler divergence as a distance, the problem \eqref{intro:mainKL} defines the so called \textit{Kullback-Leibler projection} of $\ka$ on the set $\Pi(\rho_1,\rho_2)$.

In the past years, theoretical and numerical aspects of \eqref{intro:mainKL} has been object of study in mathematical physics (e.g. \cite{CarSth,Car84,Car86, CruZam91, Fen,  NelNM, Neldyn, NelQF,  Schr31, Zam86, ZamVar86, Zam15}), probability (e.g. \cite{CatLeo94, GozLeo07, LeoGamma, LeoSurvey}), fluid mechanics (e.g. \cite{ArnCruLeoZam17, BenCarNen17}), metric geometry (e.g. \cite{GenLeoRipTam18,GigTam18}), optimal transport theory (e.g. \cite{BenCarCutNenPey2015, CarDuvPeySch, CarLab18, ChenConGeoRip, CheGeoPav, FatGozPro19, GigTamBB18, Mik04}), data sciences (e.g. \cite{FeyFXVAmaPey, GenChiBacCutPey,GenCutPey, Lui18, Lui19, PavTabTriSch, RabinPeyDelBer2011} see also the book \cite{CutPeyBook} and references therein). 

The existence of a minimizer in \eqref{intro:mainKL} was obtain in different generality by I. Czisar, L. Ruschendorf, J. M. Borwein, A. S. Lewis and R. D. Nussbaum, C. L\'eonard, N. Gigli and L. Tamanini among others \cite{BorLewNus94, Csi75, GigTam18, RusIPFP}. In the most general case the kernel $\ka$ is not even assumed to be absolutely continuous with respect to $\rho_1 \otimes \rho_2$, as opposed to our assumption \eqref{eqn:Gibbs}. In particular, under the assumption \eqref{eqn:Gibbs} on $\ka$ (see for example \cite{LeoSurvey}), a unique minimizer for \eqref{intro:mainKL} exists and $\gammaep_{opt}$ is the minimizer if and only if
\be\label{intro:SchSys}
\gammaep_{opt} = a^{\ep}(x)b^{\ep}(y)\ka, \text{ where } a^{\ep},b^{\ep} \text{ solve } 
     \ds\begin{cases}
     a^{\ep}(x)\int_{Y} b^{\ep}(y)k(x,y)d\rho_2(y) = 1 \\
     b^{\ep}(y)\int_{X} a^{\ep}(y)k(x,y)d\rho_1(x) = 1\\
     \end{cases}.
\ee

The functions $a^{\ep}(x)$ and $b^{\ep}(y)$ are called \textit{Entropic potentials}. They are unique up to the trivial transformation $a \mapsto a/\lambda$, $b \mapsto \lambda b$ for some $\lambda >0$. The system solved by the Entropic potentials is called the \emph{Schr\"{o}dinger system}. Assuming $\rho_1$ and $\rho_2$ are everywhere positive and have finite entropy with respect to $\ka$, the minimizer in \eqref{intro:mainKL} has a special form as is stated in the Theorem below \cite[Corollary 3.9]{BorLewNus94}. 

\begin{teo}[J. M. Borwein, A. S. Lewis, and R. D. Nussbaum]\label{thm:BroLewNus}
Let $(X,d_X)$ and $(Y,d_Y)$ be a Polish spaces, $c:X\times Y\to [0, \infty)$ be a bounded cost function, $\rho_1 \in \P(X)$, $\rho_2 \in \P(Y)$  be probability measures such that $\rho_1(x),\rho_2(y)>0, \forall x\in X,  y\in Y$ and $\mathcal{K} = e^{-\frac{c(x,y)}{\ep}}\rho_1\otimes\rho_2$. Then, if and $\KL(\rho_1\otimes\rho_2|\ka) <+\infty$, then for each $\ep> 0$, there exists a unique minimizer $\gamma^{\ep}\in \Pi(\rho_1,\rho_2)$ for the Schr\"odinger problem $\OTep(\rho_1,\rho_2)$ that can be written as 
\[
\gammaep_{opt}(x,y) = a^{\ep}(x)b^{\ep}(y) \ka(x,y), \quad \ln a^{\ep}(x)\in L_{\rho_1}^1(X),  \ln b^{\ep}(x) \in L_{\rho_2}^1(Y). 
\] 
\end{teo}

 \medskip


In this paper, we are interested in the following questions:
\begin{itemize}
\item[(1)] \emph{What is the regularity of the Entropic potentials $a^{\ep}$ and $b^{\ep}$?}

\item[(2)] \emph{Can we understand the structure and regularity of the minimizer in \eqref{intro:mainKL} if we consider the Schr\"odinger Bridge problem with $N$ given marginals $\rho_1,\rho_2,\dots,\rho_N$ instead of $2$?}

\end{itemize}

The answers to the questions (1) and (2) relies on the Kantorovich duality formulation of \eqref{intro:mainKL} and its extension to the multi-marginal setting: we will exploit the parallel with optimal transport to give also a new (variational) proof for the existence of a solution to the Schr\"{o}dinger system. We believe that also this contribution is important since the only available proofs of that pass through abstract results of closure of ``sum type" functions.

The multi-marginal Schr\"odinger Bridge problem, to be introduced in section \ref{sec:multimarginal}, has been recently consider in the literature from different viewpoints (e.g. \cite{BenCarCutNenPey2015, BenCarNen17, CarDuvPeySch, CarLab18, ChenConGeoRip, CutDou2014, GerGroGor19, GerKauRajEnt}) as, for instance, the Wasserstein Barycenters, Matching problem in Economics, time-discretisation of Euler Equations and Density Functional Theory in computational chemistry. 

Finally, we want to mention that G. Carlier and M. Laborde in \cite{CarLab18} show the well-posedness (existence, uniqueness and smooth dependence with respect to the data) for the multi-marginal Schr\"odinger system in $L^{\infty}$ - see \eqref{intro:SchSysMM} in section \ref{sec:multimarginal} -  via a local and global inverse function theorems. This is a different approach and orthogonal result compared to the study presented in this paper; moreover their result is restricted to measures $\rho_i$ which are absolutely continuous with respect to some reference measure, with density bounded from above and below.

\subsection*{Computational aspects and connection with Optimal Transport Theory} 

\quad  In many applications, the method of choice for numerically computing \eqref{intro:mainKL} is the so-called \textit{Iterative Proportional Fitting Procedure} (IPFP) or \textit{Sinkhorn algorithm} \cite{Sin64}. The aim of the Sinkorn algorithm is to construct the measure $\gammaep$ realizing minimum in \eqref{intro:mainKL} by fixing the shape of the guess as $\gammaep_{n} = a^n(x)  b^n(y) \ka$ (since this is the actual shape of the minimizer) and then alternatively updating either $a^n$ or $b^n$, by
matching one of the marginal distribution respectively to the target marginals $\rho_1$ or $\rho_2$, 

The IPFP sequences $(a^n)_{n\in\N}$ and $(b^n)_{n\in\N}$ are defined thus iteratively by\footnote{The iterations above system also appeared in \cite{Bac65,DemSte40, Idel16,Kru37,Yul12} with different names (eg. RAS, IPFP).}
\begin{equation}\label{eq:IPFPiteration}
\begin{array}{lcl}
\ds a^0(x) & = &1, \\
\ds b^0(y) & = &1, \\
\ds b^n(y) & = & \dfrac{1}{\int k(x,y) a^{n-1}(x)d\rho_1(x)}, \\
\ds a^n(y) & = & \dfrac{1}{\int k(x,y) b^{n}(y)d\rho_2(y)}. 
\end{array}
\end{equation}


While $a^n$ and $b^n$ will be approximations of the solution of the Schr\"{o}dinger system, the sequence of probability measures  $\gammaep_{n} = a^{n}(x) b^n(y)\K$  will approximate the minimizer $\gammaep$. 

We stress that the IPFP procedure can be easily generalized in the multi-marginal setting, whose discussion will be detailed in section \ref{sec:multimarginal}. 
\begin{itemize}

\item[(3)] \textit{Can one prove convergence of the Sinkhorn algorithm in two and several marginals case?} 

\end{itemize}

In the two marginals case, the IPFP schemes was introduced by R. Sinkhorn \cite{Sin64}. The convergence of the iterates \eqref{eq:IPFPiteration} was proven by J. Franklin and J. Lorenz \cite{FraLor89} in the discreate case and by  L. Ruschendorf \cite{RusIPFP} in the continuous one. An alternative proof based on the Franklin-Lorenz approach via the Hilbert metric was also provided by Y. Chen, T. Georgiou and M. Pavon \cite{CheGeoPav}, which is particular leads to a linear convergence rate of the procedure (in the Hilbert metric).

Despite the different approaches and theoretical guarantees obtained in the $2$-marginal problem, in the multi-marginal case the situation changes completely. Although numerical evidence suggests convergence and stability of the Sinkorhn algorithm for general class of cost functions \cite{BenCarCutNenPey2015, CutDou2014, CarDuvPeySch,CarDuvPeySch}, theoretical results guaranteeing convergence, stability were unknown (even if in \cite{RusIPFP} it is claimed that with his methods the result can be extended to the multimarginal case, but to our knowledge this has not been done yet). 

One of the contributions of this paper is to give convergence results of the Sinkhorn algorithm in the multi-marginal setting. In our approach we exploit the regularity of Entropic potentials to prove by compatness the convergence of IPFP scheme \eqref{eq:IPFPiteration}. \medskip

\noindent
\emph{Connection with Optimal Transport Theory:} the problem \eqref{intro:mainKL} allow us to create very efficient numerical scheme approximating solutions to the Monge-Kantorovich formulation of optimal transport and its many generalizations. Indeed, notice that we can rewrite \eqref{intro:mainKL} as a functional given by the Monge-Kantorovich formulation of Optimal Transport with a cost function $c$ plus an Entropic regularization parameter

\begin{align}\label{intro:compOTOTep}
\OTep(\rho_1,\rho_2) &=  \min_{\gamma\in \Pi(\rho_1,\rho_2)}\ep\int_{X\times Y}\gamma\log\left(\dfrac{\gamma}{k}\right)d(\rho_1\otimes\rho_2) \nonumber \\
&=  \min_{\gamma\in \Pi(\rho_1,\rho_2)}\int_{X\times Y}cd\gamma + \ep\int_{X\times Y}\gamma\log\gamma\, d(\rho_1\otimes\rho_2).
\end{align}

In particular, one can show that \cite{CarDuvPeySch,LeoGamma,Mik04} if $(\gammaep)_{\ep\geq 0}$ is a sequence of minimizers of the above problem, then $\gammaep$ converges when $\ep\to 0$ to a solution of the Optimal Transport $(\ep=0)$. More precisely, let us define the functionals $C_k,C_0:\P(X\times Y)\to \R\cup\lbrace +\infty\rbrace$
\[
C_{k}(\gamma) = \begin{cases}
       \int_{X\times Y}c d\gamma +\ep_k\int_{X\times Y}\rho_{\gamma}\log\rho_{\gamma}d(\rho_1\otimes\rho_2) &\text{ if } \gamma \in \Pi(\rho_0,\rho_1) \\
       +\infty &\text{ otherwise }\\ 
     \end{cases},
\] 
\[
C_{0}(\gamma) = \begin{cases}
       \int_{X\times Y}c d\gamma  &\text{ if } \gamma \in \Pi(\rho_0,\rho_1) \\
       +\infty &\text{ otherwise }\\ 
     \end{cases}.
\] 	
Then in \cite{CarDuvPeySch, LeoSurvey,Mik04} it is shows that the sequence of functionals $(C_k)_{k\in\N}$ $\Gamma-$converges to $C_0$ with respect to the weak convergence of measures. In particular the minima and minimal values are converging and so, in particular if $c(x,y) = d(x,y)^p$, then
\[
\lim_{k\to+\infty}\operatorname{OT}^p_{\ep_k}(\rho_1,\rho_2) = W^p_p(\rho_1,\rho_2),
\]
where $W_p(\rho_1,\rho_2)$ is the $p$-Wasserstein distance between $\rho_1$ and $\rho_2$,
\[
W^p_p(\rho_1,\rho_2) = \min_{\gamma\in\Pi(\rho_1,\rho_2)}\int_{X\times Y}d^p(x,y)d\gamma(x,y).
\]
\begin{figure}[http]\label{fig:plan}
	\centering
	\includegraphics[ scale=0.24]{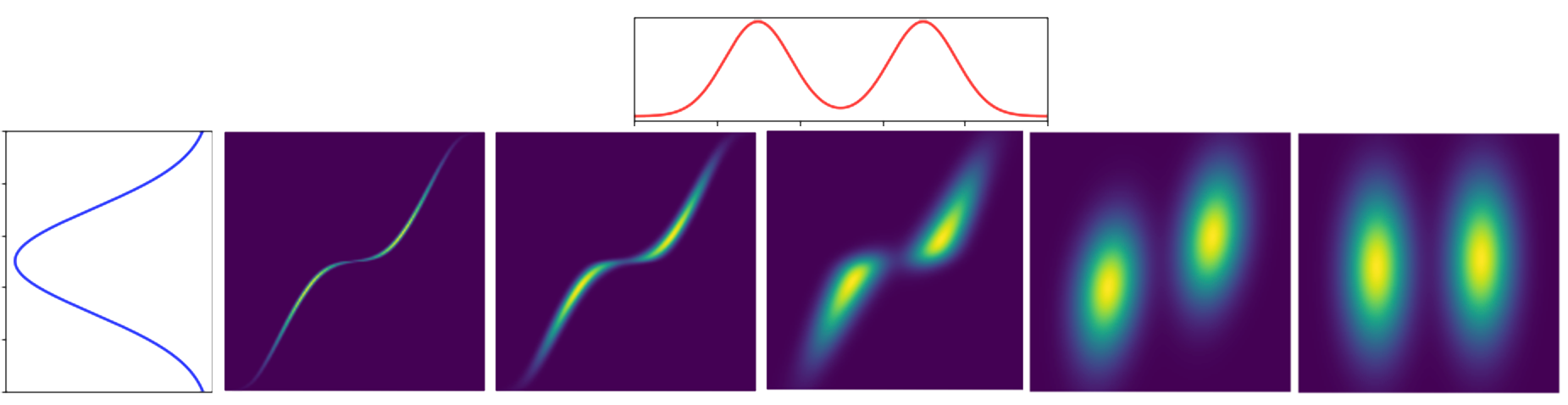}
	\caption{Support of the optimal coupling $\gammaep$ in \eqref{intro:mainKL} for the one-dimensional distance square cost with different values of $\epsilon = 10^{-1},1,10,10^{2},10^{3}$: the densities $\rho_1 \sim N(0,5)$ (blue) and $\rho_2 = \frac{1}{2}\eta_1 + \frac{1}{2}\eta_2$ is a mixed Gaussian (red), where $\eta_1 \sim N(-2,0.8)$ and $\eta_2 \sim N(2,0.7)$. The numerical computations are done using the POT library, \cite{POT}.}
	\label{fig}
\end{figure}

In the context of Optimal Transport Theory, the entropic regularization was introduced by A. Galichon and B. Salani\'e \cite{GalSal10} to solve matching problems in economics; and by M. Cuturi \cite{Cut} in the context of machine learning and data sciences. Both seminal papers received renewed attention in understanding the theoretical aspects of  \eqref{intro:compOTOTep} as well as had a strong impact in imagining, data sciences and machine learning communities due to the efficiency of the Sinkhorn algorithm.

Sinkhorn algorithm provides an efficient and scalable approximation to optimal transport. In particular, by an appropriate choice of parameters, the Sinkhorn algorithm is in fact a near-linear time approximation for computing OT distances between discrete measures \cite{AltWeeRig17}. However, as studied in \cite{Dud69, BacWee19}, the Wasserstein distance suffer from the so-called \textit{curse of dimensionality}. We refer to the recent book \cite{CutPeyBook} written by M. Cuturi and G. Peyr\'e for a complete presentation and references on computational optimal transport.

\subsection{Main contributions}

\quad In order to study the regularity of Entropic-potentials, we introduce the dual (Kantorovich) functional 
\[
D_{\ep}(u,v) = \int_X u(x)d\rho_1(x) + \int_Y v(y)d\rho_2(y) - \ep\int_{X\times Y} e^{\frac{u(x)+v(y)-c(x,y)}{\ep}}d(\rho_1\otimes\rho_2).
\]
\quad The Kantorovich duality of \eqref{intro:mainKL} is given by the following variational problem (see Proposition \ref{prop:duality2N})
\be\label{intro:entdual}
\OTep(\rho_1,\rho_2) = \sup_{u\in C_b(X),v\in C_b(Y)}D_{\ep}(u,v) +\ep.
\ee

The Entropy-Kantorovich duality \eqref{intro:entdual} appeared, for instance, in \cite{CarLab18, FeyFXVAmaPey, GerKauRajEnt, GigTamBB18,GigTam18,LeoSurvey}. The firsts contributions of this paper are (i) prove the existence of maximizers $u^*$ and $v^*$ (up to translation) in \eqref{intro:entdual} in natural spaces; (ii) show that the Entropy-Kantorovich potentials inherit the same regularity of the cost function (see the precise statement in Proposition \ref{prop:EstPot} and Theorem \ref{thm:kanto2Nmax}). 

We then link $u^*$ and $v^*$ to the solution of the Schr\"{o}dinger problem; as a byproduct of our results we are able to provide an alternative proof of the convergence of the Sinkhorn algorithm in the $2$-marginal case via a purely optimal transportation approach (Theorem \ref{thm:convIPFP}), seeing it as an alternate maximization procedure. The strength of this proof is that it can be easily generalized to the multi-marginal setting (Theorem \ref{thm:convIPFPNmarg}). 

\subsection{Summary of results and main ideas of the proofs}
\quad Our approach follows ideas from Optimal Transport and relies on the study of the duality (Kantorovich) problem \eqref{intro:entdual} of \eqref{intro:mainKL}. Analogously to the optimal transport case, if one assume some regularity (boundedness, uniform continuity, concavity) of the cost function $c$, then we can obtain the same type of regularity of the Entropy potentials $u$ and $v$. 

The relation between solution of the dual problem \ref{intro:entdual} and the Entropic-Potentials solving the Schr\"odinger system was already pointed out by C. L\'eonard \cite{LeoSurvey}. From our knowledge, the direct proof of existence of maximizers in \eqref{intro:entdual} a new result. 

Our approach to obtain the existence of Entropic-Kantorovich potentials, follow the direct method of Calculus of Variations. The key idea in the argument is to define a generalized notion of $c$-transform in the Schr\"odinger Bridge case, namely the $(c,\ep)$-transform. The main duality result, in the most general case where we assume only that $c$ is bounded, is given by the Theorem \ref{thm:kanto2Nmax} and stated below.

\begin{teo} Let $(X,d_X)$, $(Y,d_Y)$ be Polish spaces, $c:X\times Y\to \R$ be a Borel bounded cost, $\rho_1 \in \P(X)$, $\rho_2 \in \P(Y)$ be probability measures and $\ep>0$ be a positive number. Then the supremum in \eqref{intro:entdual} is attained for a unique couple $(u_0,v_0)$ (up to the trivial tranformation $(u,v) \mapsto (u+a,v-a)$). Moreover we also have 
$$u_0 \in L^{\infty}(\rho_1) \quad \text{ and } \quad v_0 \in L^{\infty}(\rho_2)$$
and we can choose the maximizers such that $\|u_0\|_{\infty} ,\|v_0\|_{\infty} \leq \frac 32 \|c\|_{\infty}$.
\end{teo}

\noindent
\emph{On the $(c,\ep)-$transform:} Given a measurable function $u:X\to\R$ such that $\int_{X}e^{u/\ep}d\mu <+\infty$, we defined the $(c,\ep)$-transform of $u$ by
\[
\ucep(y) = -\ep\log\left(\int_{X}e^{(u(x)-c(x,y))/\ep}d\mu(x)\right).
\]

One can show that this operation is well defined and, moreover, $D_{\ep}(u,\ucep) \geq  D_{\ep}(u,v), \forall u,v$ and $D_{\ep}(u,\ucep) =  D_{\ep}(u,v) \text{ if and only if } v = \ucep$ (lemma \ref{lemma:dual}). If we assume additionally regularity for the cost function $c$, for instance that $c$ is $\omega$-continuous or that it is merely bounded, the $(c,\ep)$-transform is actually a compact operator, respectively form $L^1(\rho_1) $ to $ C(Y)$ and from $L^{\infty}(\rho_1)$ to $L^p(\rho_2)$ (Proposition \ref{prop:EstPot}). \medskip

\noindent
\emph{IPFP/Sinkhorn algorithm:} As a byproduct of the above approach to the duality, we present an alternative proof of the convergence of the IPFP/Sinkhorn algorithm. The main idea in our proof is that we can rewrite the IPFP iteration substituting $a^n = \exp(u^n/\ep)$ and $b^n = \exp(v^n/\ep)$; in these new variables the iteration becomes 
\[
\begin{array}{lcl}
\ds v^n(y)/\ep = & - \log\left(\int_X k(x,y)\otimes \frac{u^{n-1}(x)}{\ep}d\rho_1\right)  \\
\ds u^n(x)/\ep  = & - \log\left(\int_Y  k(x,y)\otimes \frac{v^{n}(x)}{\ep}d\rho_2\right) 
\end{array}.
\]
Or, $v^n(y) = (u^{(n-1)})^{(c,\ep)}$ and $u^n(y) = (v^n)^{(c,\ep)}$. In particular we can interpret the IPFP in optimal transportation terms: at each step the Sinkhorn iterations \eqref{eq:IPFPiteration} are equivalent to take the $(c,\ep)$-transforms alternatively and therefore the IPFP can be seen as an alternating maximizing procedure for the dual problem. Therefore, using the aforementioned compactness, it is easy to show that $u^n$ and $v^n$ converge to to the optimal solution of the Kantorovich dual problem when $n\to+\infty$. 
\begin{teo}
Let $(X,d_X)$ and $(Y,d_Y)$ be Polish metric spaces, $\rho_1 \in \P(X)$ and $\rho_2\in\P(Y)$ be probability measures and $c:X\times Y\to \R$ be a Borel bounded cost. If $(a^n)_{n\in\N}$ and $(b^n)_{n\in\N}$ are the IPFP sequences defined in \eqref{eq:IPFPiteration}, then there exists $\lambda_n >0$ such that
\[
a^n/\lambda_n\to a \text{ in } L^p(\rho_1) \quad \text{ and } \quad \lambda_n b^n\to b \text{ in } L^p(\rho_2), \quad 1\leq p <+\infty,
\] 
for $a,b$ that solve the Schr\"{o}dinger system. In particular, the sequence $\gamma^n = a^nb^n\ka$ converges in $L^p(\rho_1\otimes\rho_2)$ to $\gammaep_{opt}$ in \eqref{intro:mainKL}, $1\leq p <+\infty$. 
\end{teo}

We recall that the argument in original proof of convergence of the Sinkhorn algorithm \cite{FraLor89} (also in \cite{CheGeoPav}) relies on defining the Hilbert metric on the projection cone of the Sinkhorn iterations. The authors show that the Sinkhorn iterates are a contraction under this metric and therefore the procedure converges. This proof has the advantage of providing automatically the rate of convergence of the iterates; however it is not easily extendable in the several marginals case. 

Our approach instead can be extended to obtain the existence and convergence results also in the multi-marginal setting:

\begin{teo}
Let $(X_i,d_{X_i})$ be Polish metric spaces and $\rho_i \in \P(X_i)$ be probability measures, for $i \in \lbrace 1,\dots, N \rbrace$ and $c:X_1\times \dots \times X_N\to \R$ be a Borel bounded cost If $(a_i^n)_{n\in\N}$ are the multi-marginal IPFP sequences  that will be defined   \eqref{eq:IPFPsequenceN}, then there exist $\lambda^n_i >0$ with $\prod_i \lambda^n_i=1$ such that
\[
a_i^n/\lambda_i^n \to a_i \text{ in } L^p(\rho_i) \quad \text{ for all } i \in \lbrace 1,\dots, N \rbrace, 1\leq p <+\infty,
\] 
where $(a_1, \ldots, a_N)$ solve the Schr\"{o}dinger system. In particular, the sequence $\gamma_N^n = \otimes^N_{i=1}a_i^n\ka$ converges in $L^p(\otimes^N_{i=1} \rho_i)$, $1\leq p <+\infty$, to the optimal coupling $\gammaep_{N, opt}$ solving the multi-marginal Schr\"odinger Bridge problem to be defined in \eqref{eq:primalSchrMult}. 
\end{teo}

\subsection{Organization of the paper}

\quad \quad The remaining part of the paper is organized as follows: Section \ref{sec:RegularityEntr} contains the main structural results of the paper, namely Proposition \ref{prop:EstPot} and Theorem \ref{thm:kanto2Nmax}. In particular, we define the main tools for showing the existence of maximizer of the Entropic-Kantorovich problem and prove regularity results of the Entropic-Kantorovich potentials via the $(c,\ep)$-transform.

In the section \ref{sec:convergenceIPFP}, we apply the above results to prove convergence of the Sinkhorn algorithm purely via the compactness argument and alternating maximizing procedue (Theorem \ref{thm:convIPFP}) and, in section \ref{sec:multimarginal}, we extend the main results of the paper to the multi-marginal Schr\"odinger Bridge problem, including convergence of Sinkhorn algorithm in the multi-marginal case (Theorem \ref{thm:convIPFPNmarg}). 

\subsection{The role of the reference measures}

\quad \quad In this subsection, we simply give a technical remark, discussing the role of the reference measures $\refmx$ and $\refmy$. We stress that all the results of the paper can be extended while considering a kind of entropic optimal transport problem, where the penalization occurs with respect to some reference measures $\refmx, \refmy$.

For $\ep>0$, we in particular may look at the problem 
\begin{align}\label{eq:Sepm}
\Sep(\rho_1,\rho_2; \refmx, \refmy) &:=  \min_{\gamma\in \Pi(\rho_1,\rho_2)} \left\{ \int_{X\times Y}c \, d\gamma + \ep  \KL( \gamma | \refmx\otimes\refmy)  \right\} \nonumber \\
&= \min_{\gamma\in \Pi(\rho_1,\rho_2)} \ep \KL ( \gamma | \mathcal{K}). 
\end{align}
where $\mathcal{K}$ is the Gibbs Kernel $\mathcal{K} = e^{-\frac{c}{\ep}}\refmx\otimes\refmy$. 

%
While having a reference measure in some situations can be quite useful (for example the Schr\"{o}dinger problem is set with $\refm_1=\refm_2 = \mathcal{L}^d$), in other it is the opposite, for example when we are considering $\rho_1, \rho_2$ to be sums of diracs. In those cases it is a much better solution to consider $\refmx=\rho_1$ and $\refmy=\rho_2$. Notice that in this case, we have that
\[
\OTep(\rho_1,\rho_2) = \Sep(\rho_1,\rho_2;\rho_1,\rho_2).
\]

Now we will see that in fact $\OTep$ is a \emph{universal} reduction for $\Sep$, meaning that we can always assume $\refmx=\rho_1$ and $\refmy=\rho_2$:

\begin{lemma} Let $(X,d,\refmx)$ and $(Y,d,\refmx)$ be a Polish metric measure spaces and $c:X\times Y\to [0,+\infty[$ be a cost function. Assume that $\rho_1 \in \P(X)$, $\rho_2 \in \P(Y)$. Then we have
$$ \Sep(\rho_1, \rho_2; \refmx, \refmy) = \OTep ( \rho_1, \rho_2) +\ep \KL(\rho_1| \refmx) +  \ep \KL(\rho_2|\refmy);$$
moreover, whenever one of the two sides is finite the minimizers of $\Sep$ and $\OTep$ are the same.
\end{lemma}

\begin{proof}
The key equality  in proving this lemma is that, whenever $\gamma \in \Gamma(\rho_1, \rho_2)$ one has
\begin{equation}\label{Eqn:equalityKL} \KL( \gamma | \refmx \otimes \refmy) = \KL(\gamma | \rho_1 \otimes \rho_2) + \KL(\rho_1 | \refmx) + \KL(\rho_2 | \refmy).
\end{equation}
While this equality is clear whenever all the terms are finite, we refer to Lemma~\ref{lem:KLsum} below for a complete proof entailing every case. From this equality we can easily get the conclusions.
\end{proof}

\begin{lemma}\label{lem:KLsum} Let $(X,\sigma_X)$ and $(Y,\sigma_Y)$ be measurable  spaces. Assume that $\gamma \in \P(X\times Y)$, $\refmx \in \P(X)$ and $\refmy\in \P(Y)$. Then we have
\begin{equation}\label{Eqn:equalityKLgen} \KL( \gamma | \refmx \otimes \refmy) = \KL(\gamma | \rho_1 \otimes \rho_2) + \KL(\rho_1 | \refmx) + \KL(\rho_2 | \refmy),
\end{equation}
where $\rho_1= (e_1)_{\sharp}\gamma$ and $\rho_2 = (e_2)_{\sharp}\gamma$ are the projections of $\gamma$ onto $X$ and $Y$ respectively.
\end{lemma}

\begin{proof}
First we assume the right hand side of \eqref{Eqn:equalityKLgen} is finite, and in particular this implies $\gamma \ac \rho_1 \otimes \rho_2$, $\rho_1 \ac \refmx$ and $\rho_2 \ac \refmy$. In particular we get $\gamma \ \refmx \otimes \refmy$ and  we can infer
$$ \frac { d \gamma}{ d (\refmx \otimes \refmy )} (x,y)= \frac { d \gamma}{ d (\rho_1 \otimes \rho_2 )} (x,y) \cdot \frac { d \rho_1}{ d\refmx } (x) \cdot \frac { d \rho_2 }{  d \refmy } (y).$$

We can now compute
\begin{align*}
 \KL( \gamma |\refmx \otimes \refmy) &= \int_{X \times Y} \ln \left( \frac { d \gamma}{ d (\refmx \otimes \refmy )} (x,y) \right) \, d \gamma \\
 &=\int_{X \times Y} \ln \left(  \frac { d \gamma}{ d (\rho_1 \otimes \rho_2 )} (x,y) \right) \, d \gamma +  \int_{X \times Y} \ln \left(\frac { d \rho_1}{ d\refmx } (x)\right) \, d \gamma + \int_{X \times Y} \ln \left(\frac { d \rho_2}{ d\refmy }(y)\right) \, d \gamma \\
 & = \KL(\gamma | \rho_1 \otimes \rho_2 ) + \int_X  \ln \left(\frac { d \rho_1}{ d\refmx } (x)\right)  \, d \rho_1 + \int_Y  \ln \left(\frac { d \rho_2}{ d\refmy } (y)\right)  \, d \rho_2 \\
 & = \KL(\gamma | \rho_1 \otimes \rho_2) + \KL(\rho_1 | \refmx) + \KL(\rho_2 | \refmy).
 \end{align*}
 
 We assume now  that the left hand side of \eqref{Eqn:equalityKLgen} is finite. Thanks to the fact that $\KL(F_{\sharp} \mu | F_{\sharp} \nu) \leq \KL( \mu | \nu)$ for every measurable function $F$, we immediately have that $ \KL(\rho_1 | \refmx)$ and $ \KL(\rho_2 | \refmy)$ are finite, and in particular $\rho_1 \ac \refmx$ and $\rho_2 \ac \refmy$. Now let us introduce $f = \frac {d \gamma}{d \refmx \otimes \refmy}$, $g_1 = \frac {d \rho_1}{d \refmx} $ and $g_2 =\frac { d \rho_2}{d \refmy}$; let us then consider any measurable set $A \subseteq X \times Y$ and assume that $(\rho_1 \otimes \rho_2) (A)=0$. In particular we have
 $$ \int_{X \times Y} \chi_A (x,y) g_1(x)g_2(y) \, d (\refmx \otimes \refmy) = (\rho_1 \otimes \rho_2) (A)=0;$$
 from this we deduce that $A$ is $\refmx \otimes \refmy$-essentially contained in the set $B=\{g_1(x) g_2(y) =0\} = B_x \cup B_y$, where $B_x= \{ g_1(x) =0\} \times Y$ and $B_y = X \times \{ g_2(y)=0\}$. However, by the marginal conditions, we have $\gamma(B_x)=\rho_1\{ g_1(x) =0\} =0$ and similarly $\gamma(B_y) =0$, which imply $\gamma(B)=0$. In particular we have
   \begin{align*}
    \gamma(A) &= \int_{X \times Y} \chi_A (x,y) f(x,y) \, d (\refmx \otimes \refmy) = \int_{X \times Y} \chi_{A \cap B} (x,y) f(x,y) \, d (\refmx \otimes \refmy) \\
    & \leq  \int_{X \times Y} \chi_{ B} (x,y) f(x,y) \, d (\refmx \otimes \refmy) \leq \gamma(B)=0.
    \end{align*}
    This proves that $\gamma \ac \rho_1 \otimes \rho_2$ and so we can perform the same calculation we did before to conclude.
\end{proof}

\section{Regularity of Entropic-Potentials and dual problem}\label{sec:RegularityEntr}

In this section we will treat the case where $c:X \times Y \to \R$ is a Borel \textit{bounded} cost; of course everything extends also to the case when $c \in L^{\infty}(\rho_1 \otimes \rho_2)$. Some of the results extend naturally for unbounded costs (for example (i), (ii), (v) in Proposition \ref{prop:EstPot}), but we prefer to keep the setting uniform.

\subsection{Entropy-Transform and a priori estimates}
\label{sec:bounded}

\quad We start by defining the Entropy-Transform. First, let us define the space $\Lexp$, which will be the natural space for the dual problem.

\begin{deff}[$\Lexp$ spaces] Let $\ep>0$ be a positive number and $(X,d_X)$ be a Polish space. We define the set $\Lexp(X,\rho_1)$ by  
\[
\Lexp(X,\rho_1) = \left\{ u:X \to [-\infty, \infty[ \, : \, u \text{ is a measurable function in } (X,\rho_1)  \text{ and } 0<\int_X e^{u/\ep} \, d \rho_1 < \infty \right\}.
\]
For $u \in \Lexp(X, \rho_1)$ we define also $\lambda_u:=\ep\log \left( \int_X e^{u/\ep} \, d \rho_1 \right)$.
\end{deff}

For simplicity, we will use the notation $\Lexp(\rho_1)$ instead of $\Lexp(X,\rho_1)$. Notice that it is possible that $u\in \Lexp(X,\rho_1)$ attains the value $-\infty$ in a set of positive measure, but not everywhere, because of the positivity constraint $\int_X e^{u/\ep} \, d \rho_1>0$. On the other hand, we have that $u \in \Lexp(X,\rho_1)$ implies $u^+ \in L^p(\rho_1)$ for every $ p \geq 1$, where $u^+(x):= \max\{ u(x),0\} $ denotes the positive part of $u$.

\begin{deff}[Entropic $c$-transform or $(c,\ep)$-transform]
Let $(X,d_X)$, $(Y,d_Y)$ be Polish spaces, $\ep>0$ be a positive number, $\rho_1 \in \P(X)$ and $\rho_2 \in \P(Y)$ be probability measures and let $c$ be a bounded measurable cost on $X \times Y$. 
The entropic $(c,\ep)$-transform  $\Fcep:\Lexp(\rho_1)\to L^0(\rho_2)$ is defined by
\begin{equation}\label{eq:F1}
 \Fcep ( u ) (y):= -\ep \log \left( \int_X e^{ \frac { u (x) - c(x,y) }{\ep}}  \, d \rho_1 (x) \right).
\end{equation}
Analogously, we define the $(c,\ep)$-transform $\Fcep:\Lexp(\rho_2)\to L^0(\rho_1)$ by
\begin{equation}\label{eq:F2}
\Fcep ( v ) (x):= -  \ep \log \left( \int_Y e^{ \frac { v (y) - c(x,y) }{\ep}}  \, d \rho_2 (y) \right).
\end{equation}
Whenever it will be clear we denote $\vcep=\Fcep(v)$ and $\ucep=\Fcep(u)$, in an analogous way to the classical $c$-transform. 
\end{deff}

Notice that $\Lexp(\rho_1)$ is the natural domain of definition for $\Fcep$ because if $u \not \in \Lexp(\rho_1)$ we would have either $\Fcep(u) \equiv - \infty$ or $\Fcep(u) \equiv +\infty$; moreover, thanks to the positivity constraint $\int_X e^{u/\ep} \, d \rho_1>0$ we also have $\Fcep(u)(y) \in \R$ almost everywhere. In fact we will  show that $\Fcep(u) \in L^{\infty}(\rho_2)$. 

We also remark that the $(c,\ep)$-transform is consistent with the $c$-transform when $\ep\to 0$:  $\ucep  \to  \max \lbrace u (x) - c(x,y) : x\in X \rbrace$, when $\ep\to 0$. In other words, $\ucep(y) = u^{c}(y) + O(\ep)$.

\begin{figure}[h]
	\centering
	\includegraphics[ scale=0.15]{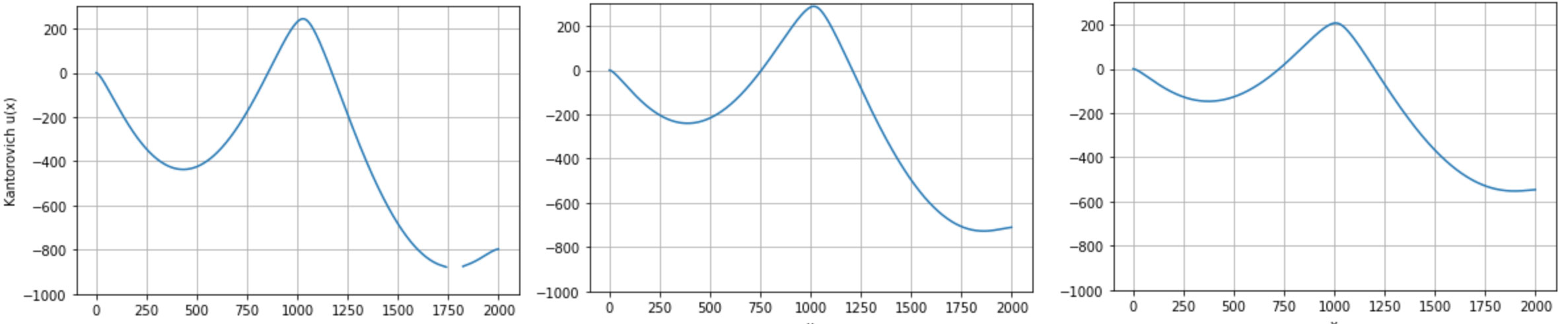}
	\caption{Entropy-Kantorovich potentials $u^{\ep}(x)-\ep \ln ( \rho_1)$ associated to the densities $\rho_1$ and $\rho_2$ for different values of the regularization parameter: $\ep_1 < \ep_2 < \ep_3$ (from left to right). The densities $\rho_1 \sim N(0,5)$ and $\rho_2 = \frac{1}{2}\eta_1 + \frac{1}{2}\eta_2$ is a mixed Gaussian, where $\eta_1 \sim N(7,0.9)$ and $\eta_2 \sim N(14,0.9)$. Notice that the values on the $y$-axis are not the same for the four figures.} 
	\label{fig}
\end{figure}

\begin{lemma}\label{lemma:F1F2} Let $(X,d_X)$, $(Y,d_Y)$ be Polish spaces, $u\in \Lexp(\rho_1),v\in \Lexp(\rho_2)$ and $\ep>0$. Then,

\begin{itemize}
\item[(i)] $\ucep(y) \in L^{\infty}(\rho_2)$ and $\vcep(x) \in L^{\infty}(\rho_1)$. More precisely,

$$-\|c\|_{\infty}-\ep \log \left( \int_X e^{ \frac { u (x)  }{\ep}}  \, d \rho_1 \right)  \leq  \ucep(y) \leq  \| c\|_{\infty} -\ep \log \left( \int_X e^{ \frac { u (x)  }{\ep}}  \, d\rho_1 \right)  $$

\item[(ii)] $\ucep(y) \in \Lexp(\rho_2)$ and $\vcep(x) \in \Lexp(\rho_1)$. Moreover $|\lambda_{\ucep} + \lambda_u| \leq \|c \|_{\infty}$.
\end{itemize}
\end{lemma}

\begin{proof}
If $u\in \Lexp(\rho_1)$ then
\begin{align*}
\ucep(y) &= -\ep \log\left(\int_X e^{\frac { u(x)-c(x,y) }{\ep}}d\rho_1\right)\\
&\leq -\ep \log\left(e^{\frac{ -\Vert c\Vert_{\infty}}{\ep}}\int_X e^{\frac { u(x)}{\ep}}d\rho_1\right)\\
&= \Vert c \Vert_{\infty} - \ep \log\left(\int_X e^{\frac { u(x)}{\ep}}d\rho_1\right). 
\end{align*}
Moreover, we get a lower bound for the above quantity using $c\geq -\|c\|_{\infty}$: 
\[
\ucep(y) = -\ep \log\left(\int_X e^{\frac { u(x)-c(x,y) }{\ep}}d\rho_1\right)
\geq -\|c\|_{\infty} -\ep \log\left(\int_X e^{\frac {u(x)}{\ep}}d\rho_1\right).
\]
Then we proved that $\ucep \in  L^{\infty}(\rho_2)$. The fact that $\vcep \in L^{\infty}(\rho_1)$ is analogous. This shows the $(i)$. Since $u\in \Lexp(\rho_1)$, by using the part $(i)$ we have that
\[
\int_Y e^{\frac{\ucep(y)}{\ep}}d\rho_2(y) \leq \int_Y e^{\Vert c\Vert_{\infty}/\ep}\left(\int_X e^{\frac {u(x)}{\ep}}d\rho_1(x)\right)^{-1}d\rho_2(y) < +\infty.
\]
Therefore $\ucep \in \Lexp(\rho_2)$ and in particular $\lambda_{\ucep} \leq -\lambda_u + \|c\|_{\infty}$; the other inequality follows with a similar calculation and the same holds for $\vcep$, which proves $(ii)$.

\end{proof}

Some of the following properties were already known for the softmax operator: for example in \cite{GenChiBacCutPey} and they are used in order to get \emph{a posteriori} regularity of the potentials but, up to our knowledge, were never used to get \emph{a priori} results.

Another very cleverly used properties of the $(c, \ep)$-transform are used in \cite{FatGozPro19} in order to obtain a new proof of the Caffarelli contraction theorem \cite{Caffarelli}.

\begin{prop}\label{prop:EstPot}
Let $\ep>0$ be a positive number, $(X,d_X)$ and $(Y,d_Y)$ be Polish metric spaces, $c:X\times Y\to[0,\infty]$ be a bounded cost  function, $\rho_1 \in \P(X)$, $\rho_2 \in \P(Y)$ be probability measures and $u\in\Lexp(\rho_1)$. Then
\begin{itemize}
\item[(i)] if $c$ is $L$-Lipschitz, then $\ucep$ is $L$-Lipschitz;
\item[(ii)] if $c$ is $\omega$-continuous, then $\ucep$ is $\omega$-continuous;
\item[(iii)] if $\vert c\vert \leq M$, then we have $|\ucep +\lambda_u|\leq M$;
\item[(iv)] if $\vert c\vert \leq M$, then $\Fcep:L^{\infty}(\rho_1)\to L^p(\rho_2)$ is a $1$-Lipschitz compact operator.
\item[(v)] if $c$ is $K$-concave with respect to $y$, then $\ucep$ is $K$-concave.

\end{itemize}

\end{prop}

\begin{proof} Of course we have that (ii) implies (i); let us prove directly (ii).
\begin{itemize}
\item[(ii)] Let $u\in\Lexp(\rho_1)$, $y_1,y_2\in Y$. We can assume without loss of generality that $\ucep(y_1)\geq \ucep(y_2)$; in that case
\begin{align*}
\vert \ucep(y_1) - \ucep(y_2)\vert &=   \ep\log\left(\int_X e^{\frac{u(x)-c(x,y_2)}{\ep}}d\rho_1\right) - \ep\log\left(\int_X e^{\frac{u(x)-c(x,y_1)}{\ep}}d\rho_1\right) \\
&=  \ep\log\left(\int_X e^{\frac{u(x)-c(x,y_1)+c(x,y_1)-c(x,y_2)}{\ep}}d\rho_1\right) - \ep\log\left(\int_X e^{\frac{u(x)-c(x,y_1)}{\ep}}d\rho_1\right) \\
&\leq \ep\log\left(e^{\frac{\omega(d(y_1,y_2))}{\ep}}\int_X e^{\frac{u(x)-c(x,y_1)}{\ep}}d\rho_1\right) - \ep\log\left(\int_X e^{\frac{u(x)-c(x,y_1)}{\ep}}d\rho_1\right) \\
&= \omega(y_1,y_2).
\end{align*}
\item[(iii)] This is a direct consequence of Lemma \ref{lemma:F1F2} (i); 

\item[(iv)] We first prove that $\Fcep$ is $1$-Lipschitz. In fact, letting $u, \tilde{u} \in L^{\infty}(\rho_1)$, we can perform a calculation very similar to what has been done in (ii): for every $y \in Y$ we have

\begin{align*} \Fcep (u) (y) &=- \ep\log\left(\int_X e^{\frac{u(x)-c(x,y)}{\ep}}d\rho_1\right) \geq  - \ep\log\left(\int_X e^{\frac{\tilde{u}(x) + \| u - \tilde u \|_{\infty} -c(x,y)}{\ep}}d\rho_1\right)
\\ &= \Fcep(\tilde{u}) (y) - \| u - \tilde{u} \|_{\infty}
\end{align*}
We can conlcude that $\| \Fcep (u) - \Fcep (\tilde{u})\|_p \leq \| \Fcep (u) - \Fcep (\tilde{u})\|_{\infty} \leq  \| u - \tilde{u} \|_{\infty}$.
This proves in particular that $\Fcep: L^{\infty}(\rho_1) \to L^p(\rho_2)$ is continuous. In order to prove that $\Fcep$ is compact it suffices to prove that $\Fcep(B)$ is precompact for every bounded set $B \subset L^{\infty}(\rho_1)$. We will use Proposition \ref{prop:compactness}; since $\Fcep$ is Lipschitz, for sure if $B$ is bounded we have that $\Fcep(B)$ is bounded in $L^p(\rho_2)$, so it remains to prove part (b) of the criterion of Proposition \ref{prop:compactness}.

Let us consider $\gamma = \rho_1\otimes\rho_2$. Since $c \in L^{\infty}(\gamma)$, by Lusin theorem we have that for every $\sigma>0$ there exists $N_{\sigma} \subset X \times Y$,  with $\gamma ( N_{\sigma})< \sigma$, such that $c|_{(N_{\sigma})^c}$ is uniformly continuous, with modulus of continuity $\omega_{\sigma}$. We now try to mimic what we did in (ii), this time keeping also track of the remainder terms we will have.

For each $y\in Y$, we consider the slice of $N_{\sigma}$ above $y$, $N^{\sigma}_y = \left\lbrace x\in X : (x,y)\in N_{\sigma} \right\rbrace$; then we consider the set of \textit{bad} $y \in Y$, where the slice $N^{\sigma}_y$ is too big:
$$N^{\sigma}_b = \left\lbrace y \in Y : \rho_1(N^{\sigma}_y)\geq \sqrt{\sigma} \right\rbrace.$$
In particular, by definition if $y\not \in  N_b^{\sigma}$ we have $\rho_1(N^{\sigma}_y) \leq \sqrt{\sigma}$, but thanks to Fubini and the condition $\gamma(N_{\sigma}) < \sigma$ we have also that $\rho_2(N^{\sigma}_b) \leq \sqrt{\sigma}$ .

Let us now consider $y, y'  \not \in N^{\sigma}_b$, and let us denote $X^*= X \setminus ( N^{\sigma}_y \cup N^{\sigma}_{y'})$. Then we have that for every $x \in X^*$, $|c(x,y)-c(x,y')| \leq \omega_{\sigma}(d(y,y'))$. We can assume without loss of generality that $\ucep(y)\geq \ucep(y')$ and we have
 	
\[
\begin{array}{lcl}
|\ucep(y)-\ucep(y')| &= - \ep \log\left(\int_X e^{(u(x)-c(x,y))/\ep}d\rho_1\right) +\ep \log\left(\int_X e^{(u(x)-c(x,y'))/\ep}d\rho_1\right) \\
&= \ep \log\left(\dfrac{\int_X e^{(u(x)-c(x,y)+c(x,y)-c(x,y'))/\ep}d\rho_1}{\int_X e^{(u(x)-c(x,y))/\ep}d\rho_1}\right) \\
&\leq \ep \log\left(e^{\frac{\omega_{\sigma}(d(y,y'))}{\ep}} +\dfrac{\int_{N^{\sigma}_y \cup N^{\sigma}_{y'}} e^{(u(x)-c(x,y'))/\ep}d\rho_1}{\int_X e^{(u(x)-c(x,y))/\ep}d\rho_1}\right) \\
&\leq \ep \log\left(e^{\frac{\omega_{\sigma}(d(y,y'))}{\ep}}+\rho_1(N^{\sigma}_y \cup N^{\sigma}_{y'}) e^{2(\Vert c\Vert + \Vert u\Vert)/\ep} \right)\\
&\leq \ep \log\left(e^{\frac{\omega_{\sigma}(d(y,y'))}{\ep}}+2\sqrt{\sigma}e^{2(\Vert c\Vert + \Vert u\Vert)/\ep}\right)
\end{array}
\] Now we denote by $A= 2 e^{2(\Vert c\Vert + \Vert u\Vert)/\ep}$ and thanks to the fact that if $a, b \geq0$ then $e^a+b \leq e^{a+b}$, we have
$$|\ucep(y)-\ucep(y')| \leq\omega_{\sigma}(d(y,y'))  + \ep \sqrt{\sigma}A \qquad \forall y, y' \not \in N_b^{\sigma}.$$

Then (having in mind also (iii) and that $A$ depends only on $\|u \|_{\infty}$), we have that also (b) of Proposition \ref{prop:compactness} is satisfied for $\Fcep(B)$, for every bounded set $B \subset L^{\infty}(\rho_1)$, granting then the compactness of $\Fcep$.

\item[(v)] In this case we are assuming that $Y$ is a geodesic space and that there exists $K \in \R$ such that for each constant speed geodetic $(y_t)_{t \in [0,1]}$ we have
$$c(x,y_t) \geq (1-t)c(x,y_0)+ t c(x,y_1) + 2 t(1-t) K d^2(y_0,y_1)  \quad \forall x \in X.$$
Then, setting $f_t(x)=e^{(u(x)-c(x,y_t))/\ep}$, the $K$-concavity inequality for $c$ implies
\begin{align*}
f_t(x) &=e^{(u(x)-c(x,y_t))/\ep}   \\
& \leq e^{(u(x)-(1-t)c(x,y_0)-tc(x,y_1)-2t(1-t)Kd^2(y_0,y_1))/\ep}  \\
 &= e^{-2 t(1-t) K d^2(y_0,y_1)/\ep} \cdot e^{((1-t)(u(x)-c(x,y_0))+t(u(x)-c(x,y_1)))/\ep} \\
 &= e^{-2 t(1-t) K d^2(y_0,y_1)/\ep} \cdot  f_0(x)^{1-t} \cdot  f_1(x)^t 
\end{align*}
Using this along with H\"{o}lder inequality we get
\begin{align*}
\ucep(y_t) & = - \ep \log\left( \int_X e^{(u(x)-c(x,y_t))/\ep}\, d\rho_1 \right)  = -\ep \log \left( \int_X f_t(x) \, d \rho_1 \right) \\
& \geq  - \ep \log\left(\int_X e^{-2 t(1-t) K d^2(y_0,y_1)/\ep} \cdot  f_0(x)^{1-t} \cdot  f_1(x)^t \,d\rho_1\right) \\
&= 2t (1-t)Kd^2(y_0,y_1)- \ep \log\left(\int_X f_0(x)^{1-t} \cdot  f_1(x)^t \,d\rho_1\right) \\
&\geq 2t (1-t)Kd^2(y_0,y_1)- \ep \log\left( \Bigl(\int_X f_0 \, d \rho_1 \Bigr)^{1-t} \cdot \Bigl(\int_X f_1 \, d \rho_1 \Bigr)^{t}\right) \\
& =2t (1-t)Kd^2(y_0,y_1) + (1-t)\ucep(y_0)+ t\ucep(y_1).
\end{align*}
\end{itemize}

\end{proof}

\begin{oss}[Entropic $c$-transform for $\Sep(\rho_1,\rho_2;\refmx,\refmy)$]\label{rmk:change} It is possible to define the entropic $c$-transform also for the Schr\"odinger problem $\Sep(\rho_1,\rho_2;\refmx,\refmy)$ with reference measures $\refmx \in \P(X)$ and $\refmy\in \P(Y)$. In this case, 
\begin{equation}\label{eq:F1alt}
\Fcep_{*} (u)  (y)= \ep\log(\rho_2(y)) -\ep \log \left( \int_X e^{ \frac { u (x) - c(x,y) }{\ep}}d\refmx(x)\right), \quad \text{and}
\end{equation}
\begin{equation}\label{eq:F2alt}
\Fcep_{*}(v) (x) = \ep\log(\rho_1(x))-\ep \log \left( \int_X e^{ \frac { u (x) - c(x,y) }{\ep}}d\refmy(y)\right).
\end{equation}

It is easy to see that 
$$ \begin{cases} \Fcep_{*}(v) &= \ep\log\rho_1 + \Fcep( v - \ep \log \rho_2) \\ \Fcep_{*}(u) &= \ep\log\rho_2 + \Fcep( u - \ep \log \rho_1) \end{cases} $$
 so that in fact the $(c,\ep)-$transforms with reference measures are in fact the $(c,\ep)$-trasforms conjugated by the addition of a function. In particular we can get exactely the same estimates we did in Lemma \ref{lemma:F1F2}, up to translate in the appropriate manner. For example we would have if $u\in \Lexp(\refmx)$, we would have then $\ucep_{*}(y) -\ep\log(\rho_2(y)) \in L^{\infty}(\refmy)$.
 
 \end{oss}

\subsection{Dual problem}

Let $u\in \Lexp(\rho_1), v\in \Lexp(\rho_2)$ and consider the Entropy-Kantorovich functional,
\begin{equation}\label{eqn:dualdef}
\Dep(u,v) =  \int_X  u(x)d\rho_1(x) + \int_Y v(y)d\rho_2(y) - \ep\int_{X\times Y} e^{\frac{u(x)+v(y)-c(x,y)}{\ep}}d(\rho_1\otimes\rho_2)
\end{equation}

What are the minimal assumption on $u,v$ in order to make sense for $\Dep(u,v)$? First of all if $u^+ \in L^1(\rho_1)$ and $v^+ \in L^1(\rho_2)$ then $\Dep(u,v)< \infty$ and in particular in order to have $\Dep(u,v)>-\infty$ we need $u\in \Lexp(\rho_1), v\in \Lexp(\rho_2)$ which is then a natural assumption (since we want to compute the supremum of $\Dep$). 

\begin{lemma}\label{lemma:dual} Let us consider $\Dep: \Lexp(\rho_1) \times \Lexp(\rho_2)\to\mathbb{R}$ defined as in \eqref{eqn:dualdef}, then
\begin{equation}\label{est:optcond}
D_{\ep}(u,\ucep) \geq  D_{\ep}(u,v) \qquad \forall v \in \Lexp(\rho_2),
\end{equation}
\begin{equation}\label{est:optcond2}
D_{\ep}(u,\ucep) =  D_{\ep}(u,v) \text{ if and only if } v = \ucep.
\end{equation}
In particular we can say that $\ucep \in {\rm argmax} \{ D_{\ep} ( u,v) \; : \; v \in \Lexp(\rho_2) \}$.
\end{lemma}

\begin{proof}
By Fubini's theorem and equation \eqref{eq:F1}, we have
\begin{align*}
\Dep(u,v) &= \int_X u(x)d\rho_1(x) + \int_Y v(y)d\rho_2(y) - \ep\int_{X\times Y} e^{\frac{u(x)+v(y)-c(x,y)}{\ep}}d(\rho_1\otimes\rho_2),\\
&= \int_X u(x)d\rho_1(x) + \int_Y v(y)d\rho_2(y) - \ep\int_Y e^{\frac{v(y)}{\ep}}\left(\int_X e^{\frac{u(x)-c(x,y)}{\ep}}d\rho_1\right)d\rho_2,\\
&= \int_X u(x)d\rho_1(x) + \int_Y v(y) - \ep e^{\frac{v(y)-\ucep(y)}{\epsilon}}d\rho_2(y).
\end{align*}  
Therefore, for any $v\in \Lexp(\rho_1)$, $ \Dep(u,v) \leq \Dep(u,\ucep)$, since the function $g(t) = t-\ep e^{(t-a)/\ep}$ is strictly concave and attains its maximum in $t=a$. In particular, $D_{\ep}(u,\ucep)= D_{\ep}(u,v)$ if and only if $v = \ucep$.
\end{proof}

\begin{lemma}\label{lemma:betterpotentials} Let us consider $u \in \Lexp(\rho_1)$ and $v \in \Lexp(\rho_2)$. Then there exist $u^* \in \Lexp(\rho_1)$ and $v^* \in \Lexp(\rho_2)$ such that
\begin{itemize}
\item $D_{\ep}(u,v) \leq D_{\ep}(u^*,v^*)$;
\item $ \| v^* \|_{\infty} \leq 3\| c\|_{\infty}/2 $;
\item $ \| u^*\|_{\infty} \leq 3\| c\|_{\infty}/2 $.
\end{itemize}
Moreover we can choose $a \in \mathbb{R}$ such that $u^*= (v+a)^{(c,\ep)}$ and $v^*=(u^*)^{(c,\ep)}$.
\end{lemma}

\begin{proof} Let us apply Proposition \ref{prop:EstPot} (iii) to $v$ and $\tilde{u}=\vcep$:
$$ - \| c\|_{\infty} \leq \vcep + \lambda_v \leq \|c \|_{\infty} $$
$$  - \| c\|_{\infty} \leq (\vcep)^{(c,\ep)}+ \lambda_{\vcep} \leq \|c \|_{\infty} $$
Let us define $\tilde{u}=\vcep$ and $\tilde{v}=(\vcep)^{(c,\ep)}$. Then by Lemma \ref{lemma:dual} we have of course that $D_{\ep}(u,v) \leq D_{\ep}(\tilde{u},\tilde{v})$; now we know that $\Dep(\tilde{u}-a, \tilde{v}+a)=\Dep(\tilde{u}, \tilde{v})$ for any $a \in \mathbb{R}$ and moreover 
$$ \| \tilde{u}-a \|_{\infty} \leq \|c\|_{\infty} + | a+\lambda_v| \qquad  \| \tilde{v}+a \|_{\infty} \leq \|c\|_{\infty} + |\lambda_{\vcep} -a|.$$
We can now choose $a^*= (\lambda_{\vcep}- \lambda_v) /2$ and, recalling Lemma \ref{lemma:F1F2} (ii) we can conclude that $u^*=\tilde{u}-a^*$ and $v^*=\tilde{v}+a^*$ satisfy the required bounds.
\end{proof}

\begin{teo}\label{thm:kanto2Nmax} Let $(X,d_X)$, $(Y,d_Y)$ be Polish spaces, $c:X\times Y\to \R$ be a Borel bounded cost, $\rho_1 \in \P(X)$, $\rho_2 \in \P(Y)$ be probability measures and $\ep>0$ be a positive number. Consider the problem
\begin{equation}\label{eq:dualitySep}
\sup \left\{ D_{\ep}(u,v) \; : \; u \in \Lexp(\rho_1) , v \in \Lexp(\rho_2) \right\}.
\end{equation}
Then the supremum in \eqref{eq:dualitySep} is attained for a unique couple $(u_0,v_0)$ (up to the trivial tranformation $(u,v) \mapsto (u+a,v-a)$). In particular we have 
$$u_0 \in L^{\infty}(\rho_1) \quad \text{ and } \quad v_0 \in L^{\infty}(\rho_2);$$
moreover we can choose the maximizers such that $\|u_0\|_{\infty} ,\|v_0\|_{\infty} \leq \frac 32 \|c\|_{\infty}$.
\end{teo}

\begin{proof}
Now, we are going to show that the supremum is attainded in the right-hand side of \eqref{eq:dualitySep}. Let $(u_n)_{n\in\N} \subset \Lexp(\rho_1) $ and $(v_n)_{n\in \N} \subset \Lexp(\rho_2)$ be  maximizing sequences. Due to Lemma  \ref{lemma:betterpotentials}, we can suppose that $u_n\in L^{\infty}(\rho_1)$, $v_n\in L^{\infty}(\rho_2)$ and $\Vert u_n\Vert_{\infty},\Vert v_n\Vert_{\infty} \leq \frac 32 \| c \|_{\infty}$. Then by Banach-Alaoglu theorem there exists subsequences $(u_{n_k})_{n_k\in\N}$ and $(v_{n_k})_{n_k\in\N}$ such that $u_{n_k}\rightharpoonup \overline{u}$ and $v_{n_k}\rightharpoonup \overline{v}$. In particular, $\tilde{u}_{n_k}+\tilde{v}_{n_k}-c\rightharpoonup \overline{u}+\overline{v}-c$.  

First, notice that since $t \mapsto e^t$ is a convex function, we have
\begin{align*}
\liminf_{n\to\infty} \int_{X\times Y}e^{\frac{u_n+v_n-c}{\ep}}d(\rho_1\otimes\rho_2) &= \liminf_{n\to\infty} \int_{X\times Y}e^{\frac{u_n+v_n-c}{\ep}}d(\rho_1\otimes\rho_2) \\
&\geq  \int_{X\times Y}e^{\frac{\overline{u}+\overline{v}-c}{\ep}}d(\rho_1\otimes\rho_2).
\end{align*}
Moreover,
\begin{align*}
\sup_{u,v} \Dep(u,v) &= \lim_{n\to\infty}\left\lbrace\int_X u_nd\rho_1 + \int_Y v_nd\rho_2 - \ep\int_{X\times Y}e^{\frac{u_n+v_n-c}{\ep}}d(\rho_1\otimes\rho_2) \right\rbrace \\
&\leq \lim_{n\to\infty}\left\lbrace\int_X u_n d\rho_1 + \int_Y v_n d\rho_2  \right\rbrace - \ep \liminf_{ n \to \infty}\left\lbrace \int_{X\times Y}e^{\frac{u_n+v_n-c}{\ep}}d(\rho_1\otimes \rho_2)\right\rbrace \\
&\leq \int_X \overline{u}d\rho_1 + \int_Y \overline{v}d\rho_2 - \ep\int_{X\times Y}e^{\frac{\overline{u}+\overline{v}-c}{\ep}}d(\rho_1\otimes\rho_2) = D(\overline{u},\overline{v}).
\end{align*}

So, $(\overline{u},\overline{v})$ is a maximizer for $\Dep$. By construction, we have also that $\overline{u} \in L^{\infty}(\rho_1)$ and $\quad \overline{v} \in L^{\infty}(\rho_2)$. Finally, the strictly concavity of $D_{\ep}$ and Lemma \ref{lemma:dual} implies that the maximizer is unique and, in particular $\overline{v} = \overline{u}^{(c,\ep)}$. 

\end{proof}

\begin{cor}
Let $(X,d_X,\refmx)$, $(Y,d_Y,\refmy)$ be Polish metric measure spaces, $c:X\times Y\to \mathbb{R}$ be a Borel bounded cost function, $\rho_1\in\P(X)$ and $\rho_2\in\P(Y)$ be probability measures such that $\KL(\rho_1|\refmx) + \KL(\rho_1|\refmy) < \infty$. Consider the dual functional $\tilde \Dep:\Lexp(\refmx)\times\Lexp(\refmy)\to \R$,
\[
\tilde{\Dep}(u,v) = \int_X u(x)\rho_1(x)d\refmx(x) + \int_Y v(y)\rho_2(y)d\refmy(y) 
- \ep\int_{X\times Y} e^{\frac{u(x)+v(y)-c(x,y)}{\ep}}d(\refmx(x)\otimes\refmy(y)).
\]
Then the supremum 
\[
\sup \left\{ D_{\ep}(u,v) \; : \; u \in \Lexp(\refmx) , v \in \Lexp(\refmy) \right\}.
\]
is attained for a unique couple $(u_0,v_0)$ and in particular we have 
$$u_0 - \ep\log\rho_1 \in L^{\infty}(\refmx) \quad \text{ and } \quad v_0 -\ep\log\rho_2 \in L^{\infty}(\refmy).$$
\end{cor}
\begin{proof}
The proof follows by the change of variable $T:(u,v)\mapsto (u -\ep \log \rho_1, v-\ep \log \rho_2)$ which is such that $\tilde \Dep (u,v)= \Dep( T(u,v))+ \ep \KL(\rho_1|\refmx) + \ep \KL(\rho_1|\refmy)$, and  Theorem \ref{thm:kanto2Nmax}. Another way is to apply same arguments of theorem \eqref{thm:kanto2Nmax} by using the Entropic $c$-transform $u^{(c,\ep)}_{\refmx}$ described in Remark \ref{rmk:change}. 
\end{proof}

In the following proposition an important concept will be that of bivariate transformation. Given $\ka$ a Gibbs measure, $a(x)$ and $b(y)$ two measurable function with respect to $\kappa$, such that $a,b \geq0$, we define the bivariate transformation of $\ka$ through $a$ and $b$ as 
\begin{equation}\label{eqn:kappa}
\kappa(a,b):= a(x)  b(y) \cdot \ka
\end{equation}
this is still a (possibily infinite) measure.

\begin{lemma}\label{lem:easydual} Let $\ep>0$ be a positive number, $(X,d_X)$ and $(Y,d_Y)$ be Polish metric spaces, $c:X\times Y\to \mathbb{R}$ be a cost function (not necessarily bounded), $\rho_1 \in \P(X)$, $\rho_2 \in \P(Y)$ be probability measures and let $\kappa$ as in \eqref{eqn:Gibbs}. Then for every $\gamma \in \Pi(\rho_1, \rho_2)$, $u \in \Lexp(\rho_1) $ and $ v \in \Lexp(\rho_2)$ then we have 
\begin{equation}\label{eqn:uvgamma}
\ep \KL(\gamma|\ka) \geq D_{\ep} ( u,v) + \ep, \quad \text{ with equality iff }\gamma =\kappa(e^{u/\ep}, e^{v/\ep}),
\end{equation}
where $\kappa$ is defined as in \eqref{eqn:kappa}.
\end{lemma}

\begin{proof}
First of all we can assume $\gamma \ll \ka$, otherwise the right hand side would be $+ \infty$ and so the inequality would be verified; then if we denote (with a slight abuse of notation) $\gamma(x,y)$ the density of $\gamma$ with respect to $\ka$, we get

\begin{align*}
\ep \KL(\gamma|\ka) &= \int_{X\times Y}c d\gamma + \ep\int_{X\times Y}\gamma\log\gamma  d\left(\rho_1\otimes\rho_2\right) \\
&= \int_{X\times Y}(c+\ep\log\gamma-u-v)\cdot\gamma d\rho_1\otimes\rho_2 + \int_X ud\rho_1 + \int_Y vd\rho_2 \\
&= \int_X u d\rho_1 + \int_Y v d\rho_2 + \int_{X\times Y}\left(\ep\log\gamma+c-u-v\right)\cdot \gamma d\left(\rho_1\otimes\rho_2\right) \\
&\geq \int_X u d\rho_1 + \int_Y v d\rho_2- \ep\int_{X\times Y} e^{\frac{u+v-c}{\ep}}d\left(\rho_1\otimes\rho_2\right) + \ep \\
&= D_{\ep}(u,v) + \ep,
\end{align*}
where we used $ts + \ep t\ln t - \ep \geq -\ep e^{-s/\ep}$, with equality if $t = e^{-s/\ep}$. Notice in particular that, as we wanted, there is equality iff $\gamma = e^{(u(x)+v(y)-c(x,y))/\ep} \cdot \rho_1 \otimes \rho_2 = \kappa(e^{u/\ep}, e^{v/\ep}) $.
\end{proof}

\begin{prop}[Equivalence and complementarity condition]\label{prop:equiv_comp}
Let $\ep>0$ be a positive number, $(X,d_X)$ and $(Y,d_Y)$ be Polish metric spaces, $c:X\times Y\to \mathbb{R}$ be a bounded cost function, $\rho_1 \in \P(X)$, $\rho_2 \in \P(Y)$ be probability measures and let $\kappa$ as in \eqref{eqn:Gibbs}. Then given $u^* \in \Lexp(\rho_1) , v^* \in \Lexp(\rho_2)$, the following are equivalent:
\begin{enumerate}
\item \emph{(Maximizers)} $u^*$ and $v^*$ are maximizing potentials for \eqref{eq:dualitySep};
\item \emph{(Maximality condition)} $\Fcep (u^*)=v^*$ and $\Fcep(v^*)=u^*$;
\item \emph{(Schr\"{o}dinger system)} let $\gamma^*=\kappa(e^{u^*/\ep}, e^{v^*/\ep})=e^{(u^*(x)+v^*(y)-c(x,y))/\ep} \cdot \rho_1 \otimes \rho_2$, then $\gamma^* \in \Pi(\rho_1, \rho_2)$; 
\item \emph{(Duality attainement) }$\OTep(\rho_1,\rho_2) = D_{\ep} (u^*,v^*) +\ep$.
\end{enumerate}
Moreover in those cases $\gamma^*$, as defined in 3, is also the (unique) minimizer for the problem \eqref{intro:mainKL}
\end{prop}

\begin{proof} 
We will prove $ 1 \Rightarrow 2  \Rightarrow 3  \Rightarrow 4  \Rightarrow 1$.

\begin{itemize}

\item[1. $\Rightarrow$ 2.] This is a straightforward application of Lemma \ref{lemma:dual}. In fact thanks to \eqref{est:optcond} we have $D_{\ep} (u^*, \Fcep(u^*)) \geq D_{\ep}(u^*,v^*)$; however, by the maximality of $u^*,v^*$ we have also $D_{\ep}(u^*,v^*) \geq D_{\ep} (u^*, \Fcep(u^*)) $, and so we conclude that $D_{\ep}(u^*,\Fcep(u^*))=D_{\ep}(u^*,v^*)$. Thanks to \eqref{est:optcond2} we then deduce that $v^*=\Fcep(u^*)$. We can follow a similar argument to prove that conversely $u^*=\Fcep(v^*)$.

\item[2. $\Rightarrow$ 3.] A simple calculation shows for every $u \in \Lexp(\rho_1) $ and $ v \in \Lexp(\rho_2)$ we have $(\pi_1)_{\sharp} ( \kappa(e^{u/\ep}, e^{v /\ep})) = e^{(u-v^{(c, \ep)})/\ep}\rho_1$ and similarly $(\pi_2)_{\sharp} ( \kappa(e^{u/\ep}, e^{v /\ep})) = e^{(v-u^{(c, \ep)})/\ep}\rho_2$. So if we assume 2. it is trivial to see that in fact $\gamma^* = \kappa(e^{u^*/\ep}, e^{v^* /\ep})  \in \Pi (\rho_1, \rho_2)$

\item[3. $\Rightarrow$ 4.] since $\gamma^* \in \Pi(\rho_1, \rho_2)$, from Lemma \ref{lem:easydual} we have
\begin{align}\label{eqn:ineq1compl}\ep \KL(\gamma^*|\ka) &\geq D_{\ep} ( u,v) + \ep \qquad  &\forall u \in \Lexp(\rho_1), v \in \Lexp(\rho_2) \\
\label{eqn:ineq2compl}
\ep \KL(\gamma|\ka) &\geq D_{\ep} ( u^*,v^*) + \ep  &\forall \gamma \in \Pi(\rho_1,\rho_2). 
\end{align}
Moreover, since by definition $\gamma^*= \kappa(e^{u^*/\ep}, e^{v^*/\ep})$, Lemma \ref{lem:easydual} assure us also that
\begin{equation}\label{eqn:ineq3compl} \ep \KL(\gamma^*|\ka) \geq D_{\ep} ( u^*,v^*) + \ep.
\end{equation}
Putting now \eqref{eqn:ineq1compl},\eqref{eqn:ineq2compl} and \eqref{eqn:ineq3compl} together we obtain
$$ \ep \KL(\gamma|\ka) \geq  D_{\ep} ( u^*,v^*) + \ep =  \ep \KL(\gamma^*|\ka) \geq  D_{\ep} ( u,v) + \ep;$$
in particular we have $\ep \KL(\gamma|\ka) \geq \ep\KL(\gamma^*|\ka)$ which grants us that $\gamma^*$ is a minimizer for \eqref{intro:mainKL} and that in particular $\OTep(\rho_1,\rho_2) =  \ep\KL(\gamma^*|\ka) = D_{\ep} ( u^*,v^*)+\ep$.

\item[4. $\Rightarrow$ 1.] Looking at \eqref{eqn:uvgamma} and minimizing in $\gamma$ we find that
$$\OTep(\rho_1,\rho_2) \geq D_{\ep} ( u,v) + \ep \qquad  \forall u \in \Lexp(\rho_1), v \in \Lexp(\rho_2);$$
using that by hypotesis $\OTep(\rho_1,\rho_2)= D_{\ep} ( u^*,v^*) + \ep$, we get that
$$D_{\ep} ( u^*,v^*)  \geq D_{\ep} ( u,v)   \qquad  \forall u \in \Lexp(\rho_1), v \in \Lexp(\rho_2),$$
that is, $u^*,v^*$ are maximizing potentials for \eqref{eq:dualitySep}.

\end{itemize}
Notice that in proving $3 \Rightarrow 4$ we incidentally proved that $\gamma^*$ is the (unique) minimizer.
\end{proof}

Finally, we conclude this section by giving a short proof of the duality between \eqref{intro:mainKL} and \eqref{eq:dualitySep}. 

\begin{prop}[General duality]\label{prop:duality2N}
Let $\ep>0$ be a positive number, $(X,d_X)$ and $(Y,d_Y)$ be Polish metric spaces, $c:X\times Y\to \mathbb{R}$ be a bounded cost function, $\rho_1 \in \P(X)$, $\rho_2 \in \P(Y)$ be probability measures. Then duality holds
\[
\OTep(\rho_1,\rho_2) = \max \left\{ D_{\ep}(u,v) \; : \; u \in \Lexp(\rho_1) , v \in \Lexp(\rho_2) \right\} +\ep.
\]
\end{prop}

\begin{proof} From Theorem \ref{thm:kanto2Nmax} we have the existence of a maximizing pair of potentials $u^*,v^*$. In particular we have
$$ \max \left\{ D_{\ep}(u,v) \; : \; u \in \Lexp(\rho_1) , v \in \Lexp(\rho_2) \right\} +\ep = D_{\ep} (u^*,v^*) +\ep;$$
this, together with point 4 in Proposition \ref{prop:equiv_comp} (which is true since 1 holds true), proves the duality. 
\end{proof}

By a similar argument, one can show that the duality holds also for the functional $\Sep(\rho_1,\rho_2;\refmx,\refmy)$.
\begin{cor}
Let $\ep>0$ be a positive number, $(X,d_X,\refmx)$ and $(Y,d_Y,\refmy)$ be Polish metric measure spaces, $c:X\times Y\to \mathbb{R}$ be a bounded cost function, $\rho_1 \in \P(X)$, $\rho_2 \in \P(Y)$ be probability measures. Then duality holds

\[
\Sep(\rho_1,\rho_2;\refmx,\refmy) = \max \left\lbrace \tilde{D}_{\ep}(u,v) : u\in\Lexp(\refmx), v\in\Lexp(\refmy) \right\rbrace + \ep.
\]
\end{cor}

\section{Convergence of the Sinkhorn / IPFP Algorithm}
\label{sec:convergenceIPFP}

\quad \quad In this section, we give an alternative proof for the convergence of the Sinkhorn algorithm. The aim of the Iterative Proportional Fitting Procedure (IPFP, also known as Sinkorn algorithm) is to construct the measure $\gammaep$ realizing minimum in \eqref{intro:mainKL} by alternatively matching one marginal distribution to the target marginals $\rho_1$ and $\rho_2$: this leads to the construction of the IPFP sequences $(a^n)_{n\in\N}$ and $(b^n)_{n\in\N}$, defined in \eqref{eq:IPFPiteration}.

We now look at the new variables $u_n := \ep \ln(a^n)$ and $v_n ;= \ep \ln(b^n)$: we can then rewrite the system \eqref{eq:IPFPiteration} as 
\[
\begin{array}{lcl}
\ds v_n(y)/\ep = & - \log\left(\int_X k(x,y)e^{ \frac{u_{n-1}(x)}{\ep}}d\rho_1\right)  \\
\ds u_n(x)/\ep  = & - \log\left(\int_Y k(x,y)e^{ \frac{v_{n}(y)}{\ep}}d\rho_2\right) 
\end{array}.
\]
In other words, using the $(c,\ep)-$transform and the expression of $k$ given in \eqref{eqn:Gibbs}, $v^n(y) = (u^{(n-1)})^{(c,\ep)}$ and $u^n(y) = (v^n)^{(c,\ep)}$.

\begin{teo}\label{thm:convIPFP}
Let $(X,d_X)$ and $(Y,d_Y)$ be Polish metric spaces, $\rho_1 \in \P(X)$ and $\rho_2\in\P(Y)$ be probability measures and $c:X\times Y\to \R$ be a Borel bounded cost. If $(a^n)_{n\in\N}$ and $(b^n)_{n\in\N}$ are the IPFP sequences defined in \eqref{eq:IPFPiteration}, then there exists a sequence of positive real numbers $(\lambda^n)_{n \in \N}$ such that
\[
a^n/\lambda^n\to a \text{ in } L^p(\rho_1) \quad \text{ and } \quad \lambda^nb^n \to b \text{ in } L^p(\rho_2), \quad 1\leq p <+\infty,
\] 
where $(a,b)$ solve the Schr\"{o}dinger problem. In particular, the sequence $\gamma^n = a^nb^n k$, where $k$ is defined in \eqref{eqn:Gibbs} converges in $L^p(\rho_1\otimes\rho_2)$ to $\gammaep_{opt}$, the density of the minimizer of \eqref{intro:mainKL} with respect to $\rho_1 \otimes \rho_2$, for any $1\leq p <+\infty$. 
\end{teo}

\begin{proof}
Let $(a^n)_{n\in\N}$ and $(b^n)_{n\in\N}$ be the IPFP sequence defined in \eqref{eq:IPFPiteration}. Let us write $a^n = e^{u_n/\ep}$, $b^n = e^{v_n/\ep}$; then, in this new variables, we noticed that the iteration can be written with the help of the $(c, \ep)$-transform:
\[
\begin{cases}
v_{2n+1} = (u_{2n})^{(c,\ep)} \\
u_{2n+1} = u_{2n} \\ 
     \end{cases}, \quad \begin{cases}
v_{2n+2} =  v_{2n+1} \\
u_{2n+2} = (v_{2n+1})^{(c,\ep)} \\ 
     \end{cases}.
\] 	
Notice that, as soon as $n\geq2$, we have $u_n \in L^{\infty}(\rho_1)$ and $v_n \in L^{\infty}(\rho_2)$ thanks to the regularizing properties of the $(c,\ep)$-transform proven in Lemma \ref{lemma:F1F2} and, moreover, thanks to \eqref{est:optcond}  and Proposition \ref{prop:duality2N} we have 
\[
\Dep(u_n,v_n)\leq \Dep(u_{n+1},v_{n+1}) \leq \dots \leq \OTep(\rho_1,\rho_2) - \ep.
\]
Then, by the same argument used in the proof of Lemma \ref{lemma:betterpotentials} it is easy to prove that there for each $n \geq 2$ there exists $\ell_n \in \R$ such that $ \|u_n - \ell_n \|_{\infty}, \|v_n +\ell_n\| \leq \frac 32 \|c\|_{\infty}$. Now, thanks to Proposition \ref{prop:EstPot} we have that the sequeces $u_n - \ell_n$ and $v_n +\ell_n$ are precompact in every $L^p$, for $1 \leq p < \infty$; in particular let us consider any limit point $u,v$. Then we have a subsequence $u_{n_k},v_{n_k}$ such that $u_{n_k}\to u$,$v_{n_k}\to v$ in $L^{\infty}$ and $u_{n_k+1} = (v_{n_k})^{(c,\ep)}$ (or the opposite). Using the continuity in $L^p$ of the $(c,\ep)$-transform, and the fact that an increasing and bounded sequence has vanishing increments, we obtain
\[
\Dep(\vcep,v) - \Dep(u,v) = \lim_{n_k\to\infty} \Dep(u_{n_k+1},v_{n_k+1}) - \Dep({u_{n_k}},v_{n_k}) = 0.
\]   
In particular, by \eqref{est:optcond2}, we have $u=\vcep$. Analogously, we obtain that $v=\ucep$ by doing the same calculation using the potentials $(u_{n_{k+2}}, v_{n_{k+2}})$ and then 
\[
\Dep(u,\ucep) - \Dep(u,v) = \lim_{n_k\to\infty} \Dep(u_{n_k+2},v_{n_k+2}) - \Dep({u_{n_k}},v_{n_k}) = 0.
\]
Now we can use Proposition \ref{prop:equiv_comp}: the implication $2 \Rightarrow 1$ proves that $(u,v)$ is a maximizer\footnote{in order to prove that there is a unique limit point at this stage, it is sufficient to take $\ell_n$ that minimizes $\| u_n - \ell_n - u\|_2$.}. In particular $a=e^{u/\ep}$, $b=e^{v/\ep}$ are solutions of the Schr\"{o}dinger equation and taking $\lambda^n=e^{\ell_n/\ep}$ we get the convergence result for $a^n$ and $b^n$, using that the exponential is Lipschitz in bounded domains.

In order to prove also the convergence of the plans, it is sufficient to note that for free we have $u_n+v_n \to u+v$ in $L^p(\rho_1 \otimes \rho_2)$, since now the translations are cancelled. Again, the fact that the exponential is Lipschitz on bounded domains and the boundedness of $k$, will let us conclude that in fact $\gamma^n \to \gamma$ in $L^p(\rho_1 \otimes \rho_2)$ for every $1 \leq p < \infty$.

\end{proof}

\begin{oss}
Notice that as long as we have more hypothesis on the smoothness of the cost function $c$ we can use precompactness of the sequences $u_n-\ell_n$ and $v_n +\ell_n$ on larger space, obtaining faster convergence. For example if $c$ is uniformly continuous we will get the uniform convergence instead of strong $L^p$ convergence.
\end{oss}

\section{Multi-marginal Schr\"odinger Problem}
\label{sec:multimarginal}

\quad In this section we generalize the results obtain previously for the Schr\"odinger problem with more than two marginals, including a proof of convergence of the Sinkhorn algorithm in the several marginals case.

We consider $(X_1,d_1), \dots, (X_N,d_N)$ Polish spaces, $\rho_1,\dots,\rho_N$ probability measures respectively in $X_1,\dots,X_N$ and $c:X_1\times\dots\times X_N\to \R$ a bounded cost. Define $\rhoN = \rho_1\otimes\dots\otimes\rho_N$ by the product measure. For every $\gamma \in \mathcal{M}(X_1\times\dots\times X_N)$, the relative entropy of $\gamma$ with respect to the \textit{Gibbs Kernel} $\K(x_1,\dots,x_N) = k^N(x_1,\dots,x_N)\rho^N = e^{-\frac{c(x_1,\dots,x_N)}{\ep}}d\rho_1\otimes\dots\otimes \rho_N$ is defined by
\begin{equation}\label{eq:defKL}
\ds\KLN(\gamma|\K) =\begin{cases}  \ds\int_{X_1\times\dots\times X_N}\gamma\log\left(\frac{\gamma}{k^N}\right)d\rhoN \qquad & \text{ if }\gamma \ll \rhoN \\ +\infty & \text{ otherwise.} \end{cases}
\end{equation}

An element $\gamma \in \Pi(\rho_1,\dots,\rho_N)$ is called coupling and is a probability measure on the product space $X_1\times \dots\times X_N$ having the $ith$-marginal equal to $\rho_i$, i.e $\gamma \in \mathcal{P}(X_1\times\dots\times X_N)$ such that $(e_i)_{\sharp}\gamma = \rho_i, \, \forall i \in \lbrace 1,\dots, N\rbrace$. 

The Multi-marginal Schr\"odinger problem is defined as the infimum of the Kullback-Leibler divergence $\KLN(\gamma|\K)$ over the couplings $\gamma \in \Pi(\rho_1,\dots,\rho_N)$
\begin{equation}\label{eq:primalSchrMult}
\OTNep(\rho_1,\dots,\rho_N) = \inf_{\gamma\in\Pi(\rho_1,\dots,\rho_N)}\ep\int_{X_1\times\dots\times X_N}\KL(\gamma|\K)d\gamma.
\end{equation} 

Optimal Transport problems with several marginals or its entropic-regularization appears, for instance, in economics G. Carlier and I. Ekeland \cite{CarEke}, and P. Chiappori, R. McCann, and N. Nesheim \cite{ChiMcCNes}; imaging (e.g. \cite{CutDou2014,SolPeyCut2015}); and in theoretical chemistry (e.g. \cite{DMaGerNenGorSei,GerGroGor19, GorSeiVig}). The first important instance of such kind of problems is attributed to Brenier's generalised solutions of the Euler equations for incompressible fluids \cite{Bre89, Bredual93, BreMin99}. 

We point out that the entropic-regularization of the multi-marginal transport problem leads to a problem of multi-dimensional matrix scaling \cite{FraLor89,Rag84}.  An important example in this setting is the Entropy-Regularized \textit{Wasserstein Barycenter} introduced by M. Agueh and G. Carlier in \cite{AguCar2011}. The Wasserstein Barycenter defines a non-linear interpolation between several probabilities measures generalizing the Euclidian barycenter and turns out to be equivalent to Gangbo-\'Swie\c{c}h cost \cite{GaSw}, that is $c(x_1,\dots,x_N) = \frac{1}{2}\Vert x_j-x_i\Vert^2$ . 

In the next section we extend to the multi-marginal setting the notions and properties of the Entropy $c$-transform done in section \ref{sec:RegularityEntr}. As a consequence, we generalise the proof of convergence of IPFP.
 
\subsection{Entropy-Transform}
 
\quad Analogously to definitions \eqref{eq:F1} and \eqref{eq:F2} in section \ref{sec:bounded}, we define the following Entropy $c$-transforms $\uunocep,\uduecep,\dots,\uNcep$. Notice that  the notation $\hat{u_i}$ stands for $\hat{u_i}= (u_1,\dots,u_{i-1},u_{i+1},\dots,u_N)$.

\begin{deff}[Entropic $c$-transform or $(c,\ep)$-transform]
Let $i\in \lbrace 1,\dots, N\rbrace$ and $\ep>0$ be a positive number. Consider $(X_i,d_{X_i})$ Polish spaces, $\rho_i \in \P(X_i)$ probability measures and let $c$ a bounded measurable cost on $X_1 \times \dots \times  X_N$. 
For every $i$, the Entropy $c$-transform $\uicep$ is defined by the functional $\Ficep: \prod_{j \neq i} \Lexp(\rho_{j})\to L^0(\rho_i)$,
\begin{equation}
\uicep (x_i)=\Ficep( \hat u_i ) (x_{i}) = - \ep\log\left(\int_{\prod_{j\neq i}X_j}e^{\frac{\sum_{j\neq i}u_j(x_j)-c(x_1,\dots,x_N)}{\ep}} d\left(\otimes_{j\neq i}\rho_j\right)\right).
\end{equation}
In particular, we have $\uicep\in \Lexp( \rho_{i})$. For $u_i \in \Lexp(X_i,\rho)$, we denote the constant $\lambda_{u_i}$ by
\[
\lambda_{u_i} = \ep\log\left(\int_{\prod_{j\neq i}X_j}e^{\frac{\sum_{j\neq i}u_j(x_j)}{\ep}} d\left(\otimes_{j\neq i}\rho_j\right)\right).
\]
\end{deff}
There is also the possibility to reconduce us to the case $N=2$: notice that if one considers the spaces $X_i $ and $Y_i=\Pi_{j \neq i} X_j$, then $c$ is also a natural cost function on $X_i \times Y_i$. We can  then consider $\rho_i$ as a measure on $X_i$ and $\otimes_{j \neq i} \rho_j$ as a measure on $Y_i$. In this way we able to construct an entropic $c$-trasform $\Fcep$ associated to this $2$-marginal problem and it is clear that 
$$ \Ficep ( \hat u_i ) = \Fcep \Bigl( \sum_{j \neq i} u_j \Bigr).$$ 

The following lemma extend lemma \ref{lemma:F1F2} in the multi-marginal setting. We omit the proof since it follow by similar arguments.

\begin{lemma}\label{lemma:entropytransbound} For every $i\in \lbrace 1,\dots,N\rbrace$, the Entropy $c$-transform $\uicep$ is well defined. Moreover,
\begin{itemize}
\item[(i)] $\uicep \in L^{\infty}\left(\rho_i\right)$. In particular,

\begin{align*}
-\| c \|_{\infty}- \ep \log \left( \int_{\prod_{j\neq i} X_i} e^{ \frac { \sum_{j\neq i}u_j (x_j)}{\ep}}  \, d \left(\otimes_{j\neq i}\rho_j\right)\right)  &\leq  \uicep(x_{i}) \leq \\
&\leq  \| c\|_{\infty} -\ep\log\left(\int_{\prod_{j\neq i} X_i} e^{ \frac { \sum_{j\neq i}u_j (x_j)}{\ep}}  \, d \left(\otimes_{j\neq i}\rho_j\right)\right).
\end{align*}
\item[(ii)] $\uicep \in \Lexp\left(\rho_i\right)$.

\item[(iii)] \begin{equation} \label{eqn:lambdau}
 |  \uicep  (x_i)+  \sum_{i \neq j} \lambda_{u_j} | \leq \| c \|_{\infty}.
\end{equation}

\item[(iv)] if $c$ is $L$-Lipschitz (resp. $\omega$-continuous), then $\uicep$ is $L$-Lipschitz (resp. $\omega$-continuous);
\item[(v)] if $\vert c\vert \leq M$, then $\operatorname{osc}(\uicep)\leq 2M$ and $\Ficep : \prod_{j \neq i} L^{\infty}(\rho_j)\to L^p(\rho_i)$ for $i=1, \ldots, n$ are compact operators for every  $1 \leq p < \infty$.
\end{itemize}
\end{lemma}

\subsection{Entropy-Kantorovich Duality}

\quad \quad We introduce the dual functional dual function $D^N_{\ep}:\Lexp(\rho_1)\times\dots\times \Lexp(\rho_N)\to[0,+\infty]$,
\begin{equation}\label{eqn:defDN}
D^N_{\ep}(u_1,\dots,u_N) = \sum^N_{i=1}\int_{X_i}u_id\rho_i - \ep\int_{X_1\times\dots\times X_N}e^{\frac{\sum^N_{i=1}u_i(x_i)-c(x_1,\dots,x_N)}{\ep}}d\left(\rho_1\otimes\dots\otimes\rho_N\right).
\end{equation}

In the sequel we will use the invariance by translation of the dual problem, and thus we introduce the following projection operator:

\begin{lemma}\label{lem:P} Let us consider the operator $P:  \prod_{i=1}^N L^{\infty}( \rho_i) \to \prod_{i=1}^N L^{\infty}( \rho_i)  $ defined as
$$P_i(u) = \begin{cases}  u_i - \lambda_{u_i} \qquad \qquad &\text{ if } i=1, \ldots, N-1 \\
 u_i+ \sum_{j \neq i}^{N-1} \lambda_{u_j}  & \text{ if } i=N. \end{cases}$$
 Then the following properties hold
 \begin{itemize}
 \item[(i)] $D^N_{\ep}(P(u))=D^N_{\ep} (u)$;
 \item[(ii)] $\| P_i(u)\|_{\infty} \leq osc ( u_i) +  |\sum_{i=1}^N \lambda_{u_i}| $, for all $i=1, \ldots, N$;
 \item[(iii)] let $v=P(u)$. Then $u_i= \Ficep ( \hat u_i )$ if only if $v_i= \Ficep ( \hat v_i )$.
 \end{itemize}
\end{lemma}

\begin{proof} \begin{itemize}

\item[(i)] In order to prove $D^N_{\ep}(P(u))=D^N_{\ep} (u)$ we first observe that 
$$ \sum_{i=1}^N P_i ( u ) (x_i) =u_N(x_N) + \sum_{i=1}^{N-1} \lambda_{u_i} +  \sum_{i=1}^{N-1} (u_i(x_i) - \lambda_{u_i}  ) = \sum_{i=1}^N u_i(x_i).$$
In particular we have (here we denote $X=X_1 \times \cdots \times X_N$
\begin{align*}
 D^N_{\ep}(P(u)) &= \sum^N_{i=1}\int_{X_i}P_i(u)d\rho_i - \ep\int_{X_1\times\dots\times X_N}e^{\frac{\sum^N_{i=1}P_i(u)(x_i)-c(x_1,\dots,x_N)}{\ep}}d\left(\rho_1\otimes\dots\otimes\rho_N\right) \\
 & = \int_{X}\sum^N_{i=1} P_i(u) (x_i) \, d \rho^N - \ep\int_{X}e^{\frac{\sum^N_{i=1}P_i(u)(x_i)-c(x_1,\dots,x_N)}{\ep}}d \rho^N \\
 &= \int_{X}\sum^N_{i=1} u_i (x_i) d \rho^N - \ep\int_{X}e^{\frac{\sum^N_{i=1}u_i(x_i)-c(x_1,\dots,x_N)}{\ep}}d \rho^N= D^N_{\ep}(u)
\end{align*}

\item[(ii)] The inequality is not trivial only if $u_i \in L^{\infty}(\rho_i)$. In this case obviously we have $\inf u_i \leq \lambda_{u_i} \leq \sup u_i$ and in particular
$$- osc(u_i) = \inf u_i - \sup u_i \leq u_i(x_i) - \lambda_{u_i} \leq \sup u_i - \inf u_i = osc(u_i),$$
that is $\| u_i - \lambda_{u_i}\|_{\infty} \leq osc(u_i)$. This proves already the bound for $i <N$; for $i=N$ we have, letting $\lambda= \sum_{i=1}^N \lambda_{u_i}$
$$\| P_N(u ) \|_{\infty} = \| u_N  - \lambda_{u_N} + \sum_{i=1}^N \lambda_{u_i}\|_{\infty} \leq  \|  u_N  - \lambda_{u_N}\|_{\infty}  + | \lambda | \leq osc(u_N) +  | \lambda | $$

\item[(iii)] This is obvious from the fact that $\Ficep ( \widehat{ u_i - \lambda_i} )= \Ficep(\hat u_i) + \sum_{j \neq i} \lambda_j$.

\end{itemize}
\end{proof}

This projection operator allows us to generalize Lemma \ref{lemma:betterpotentials}:

\begin{lemma}\label{lemma:betterpotentials_multi} Let us consider $u_i \in \Lexp(\rho_i)$, for every $i=1, \ldots, N$. Then there exist $u_i^* \in \Lexp(\rho_i)$ for $i=1, \ldots , N$  such that
\begin{itemize}
\item $D^N_{\ep}(u_1, \ldots, u_N) \leq D^N_{\ep}(u_1^*,\ldots, u_N^*)$;
\item $ \| u_i^* \|_{\infty} \leq 3\| c\|_{\infty} $ for every $i=1, \ldots, N$.
\end{itemize}
\end{lemma}

\begin{proof} Let us construct the following sequence of potentials:

\[
\begin{cases}
u_1^{1} = \Ficep \bigl(\,\widehat{u_1}\,\bigr) \\
u_2^{1} = u_2 \\ 
u_3^{} = u_3 \\ 
\, \cdots \\
u_N^{1} = u_N \\ 
     \end{cases}, \quad \begin{cases}
u^{2}_1 =  u_1^{1} \\
u_{2} = \Ficep \bigl(\,\widehat{u^{1}_2}\,\bigr) \\
u^{2}_3 =  u_3^{1} \\ 
\, \cdots  \\
u^{2}_N =  u_N^{1} \\ 
     \end{cases}, \dots \, , \quad \begin{cases}
u_1^{N} = u_1^{N-1} \\ 
u_2^{N} = u_2^{N-1} \\ 
u_3^{N} = u_3^{N-1} \\ 
\, \cdots \\
u^{N}_N = \Ficep \bigl(\,\widehat{u^{N-1}_N}\,\bigr). \\
     \end{cases}
\] 	
Then let us consider $u^*=P(u^N)$. First of all we notice that, using the multimarginal analogous of Lemma \ref{lemma:dual} we have
$$D_{\ep}^N(u_1, \ldots, u_N) \leq D_{\ep}^N(u^1_1, \ldots, u^1_N)  \leq \cdots \leq D_{\ep}^N(u^N_1, \ldots, u^N_N)= D_{\ep}^N(u^*_1, \ldots, u^*_N).$$
Then is clear by construction that for every $i=1, \ldots, N$ we have $u^N_i=u^i_i$ and in particular, by Lemma \ref{lemma:entropytransbound} (iv) we have $osc(u^N_i) \leq 2\| c \|_{\infty}$. 
Moreover, thanks to \eqref{eqn:lambdau} it is easy to see that $ | \sum_i \lambda_{u_i^N}| \leq \| c\|_{\infty}$. Now we can use Lemma \ref{lem:P} (ii) to conclude that in fact $\| u^*_i\| \leq 3 \| c \|_{\infty}.$

\end{proof}

Similarly to Theorem \ref{thm:kanto2Nmax}, Proposition \ref{prop:equiv_comp} and \ref{prop:duality2N}, the next theorem and the following Proposition state the existence of a maximizer and the Entropic-Kantorovich duality to the multi-marginal case, along with the complementarity conditions. Since the proofs follows the same lines of the case $N=2$, without big changes, we will omit them. 

\begin{teo}\label{thm:dualNmarg}
For every $i\in\lbrace 1,\dots,N\rbrace$, let $(X_i,d_i)$ be Polish metric spaces, $\rho_i \in \P(X_i)$ be a probability measures and $c:X_1\times\dots\times X_N\to \R$ be a bounded cost function. Then for every $\ep>0$, 

\begin{itemize}
\item[(i)] The dual function $D^N_{\ep}$ is well defined on its definition domain and moreover
\begin{equation}\label{est:optcondmm}
D^N_{\ep}(\uunocep, \hat{u_1}) \geq  D^N_{\ep}(u_1,\dots,u_N), \qquad \forall \, u_i \in \Lexp(\rho_i),
\end{equation}
\begin{equation}\label{est:optcond2mm}
D^N_{\ep}(\uunocep,\hat{u_1}) =  D^N_{\ep}(u_1,\dots,u_N) \text{ if and only if } u_1 = \uunocep.
\end{equation}

\item[(ii)] The supremum is attained, up to trivial transformations, for a unique $N$-tuple $(u^0_1,\dots,u^0_N)$ and in particular we have 
$$u^0_i\in L^{\infty}(\rho_i) ,\quad \forall i\in\left\lbrace 1,\dots,N \right\rbrace. $$
Moreover if we consider $\gamma^{0,N} = e^{ (\sum_i u^0_i(x_i) - c ) / \ep } \rho^N$ then $\gamma^{0,N}$ is the minimizer of \eqref{eq:primalSchrMult}
%

\item[(iii)] Duality holds:
\[
\OTNep(\rho_1,\dots,\rho_N) = \sup\left\{ D^N_{\ep}(u_1,\dots,u_N) \; : \; u_i \in \Lexp(\rho_i), i \in \left\lbrace 1,\dots, N \right\rbrace \right\} + \ep.
\]
\end{itemize}
\end{teo}

Finally, the result extends the main results of the previous section to the multi-marginal case.

\begin{prop}[Equivalence and complementarity condition]\label{prop:equiv_comp_multi}
Let $\ep>0$ and for every $i\in\lbrace 1,\dots,N\rbrace$, let $(X_i,d_i)$ be Polish metric spaces, $\rho_i \in \P(X_i)$ be a probability measures and $c:X_1\times\dots\times X_N\to \R$ be a bounded cost function. Then given $u_i^* \in \Lexp(\rho_i) $ for every $i=1, \ldots, N$, the following are equivalent:
\begin{enumerate}
\item \emph{(Maximizers)} $u_1^*, \ldots, u_N^*$ are maximizing potentials for \eqref{eqn:defDN};
\item \emph{(Maximality condition)} $\Ficep (\hat{u_i^*})=u_i^*$ for every $i=1, \ldots, N$;
\item \emph{(Schr\"{o}dinger system)} let $\gamma^*=e^{(\sum_i u_i^*(x_i)-c)/\ep} \cdot  \rho^N $, then $\gamma^* \in \Pi(\rho_1, \ldots, \rho_N)$; 
\item \emph{(Duality attainement) }$\OTNep(\rho_1,\ldots, \rho_N) = D^N_{\ep} (u_1^*, \ldots, u_N^*) +\ep$.
\end{enumerate}
Moreover in those cases $\gamma^*$, as defined in 3, is also the (unique) minimizer for the problem \eqref{eq:primalSchrMult}
\end{prop}

The part 3. in proposition \ref{prop:equiv_comp_multi} have been already shown in different settings by J.M Borwein, A.S. Lewis and R.D. Nussbaum (\cite[Theorem 4.4]{BorLew92}, see also \cite[section 3]{BorLewNus94}) and G. Carlier \& M. Laborde \cite{CarLab18}. Our approach, being purely variational, allows us to study the convergence of the Sinkhorn algorithm in the several marginal case in a similar way of done in the previous section. 

In fact, $\OTNep(\rho_1,\dots,\rho_N)$ defines an unique element $\gammaep_{N,opt}$ - the $\KLN$-projection on $\Pi(\rho_1,\dots,\rho_N)$ - which has product density $\Pi^N_{i=1}a_i$ with respect to the Gibbs measures $\K$, where $a_i = e^{u^*_i/\ep}$ as defined in proposition \ref{prop:equiv_comp_multi}. Also in this case, an equivalent system to \eqref{intro:SchSys} can be implicity written: $\gammaep_{N,opt}$ is a solution of \eqref{eq:primalSchrMult} if and only if $\gammaep_{N,opt} = \otimes^N_{i=1}a^{\ep}(x_i)\ka, \text{ where } a^{\ep}_i \text{ solve } $
\be\label{intro:SchSysMM}
     \ds
     a^{\ep}_i(x)\int_{Y} \otimes^N_{j\neq i}a_j(x_j)k(x_1,\dots,x_N)d\rho_i(x_i) = 1, \quad \forall i = 1,\dots, N.
\ee
Therefore, by using the marginal condition $\gammaep \in \Pi(\rho_1,\dots,\rho_N)$, the functions $a_i$ can be implicitly computed
\[
a_i(x_i) = \dfrac{1}{\int_{\Pi^N_{j\neq i}X_j}\otimes^N_{j\neq i}a_j(x_j)k(x_1,\dots,x_N)d(\otimes^N_{j\neq i}\rho_j)}, \quad \forall i \in \left\lbrace 1,\dots,N\right\rbrace.
\]

\subsection{Convergence of the IPFP / Sinkhorn algorithm for several marginals}
\label{sec:convergenceIPFPMM}

\quad \quad The goal of this subsection is to prove the convergence of the IPFP/Sinkhorn algorithm in the multi-marginal setting. Analogously to \eqref{eq:IPFPiteration}, define recursively the sequences $(a^n_j)_{n\in\N}, j\in \lbrace 1,\dots,N\rbrace$ by
\begin{equation}\label{eq:IPFPsequenceN}
\begin{array}{lcl}
\ds a_1^0(x_1) & = & 1, \\
\ds a^0_j(x_j) & = & 1, \quad j \in \lbrace 2,\dots,N \rbrace, \\
\ds a^n_j(x_j) & = & \dfrac{1}{\int \otimes^N_{i<j}a_i^n(x_i)\otimes^N_{i> j}a_i^{n-1}(x_i)k^N(x_1,\dots,x_N)d(\otimes^N_{i\neq j}\rho_i)}, \, \forall n\in \N.
\end{array}
\end{equation}

Also here,  by writing $a^n_j = \exp(u^n_j/\ep)$, for all $j\in\lbrace 1,\dots,N\rbrace$, one can rewrite the IPFP sequences \eqref{eq:IPFPsequenceN} in terms of Entropic $(c,\ep)$-transforms,
\begin{align*}
u^n_j(x_j) &= - \ep\log\left(\int_{\Pi_{i \neq j}X_i} k^N(x_1,\dots,x_N)\otimes_{i\neq j}e^{u^n_i(x_i)/\ep}d\left(\otimes^N_{i\neq j}\rho_i\right)\right) \\
&= (\hat{u^n_j})^{(N,c,\ep)}(x_j).
\end{align*}

Then, the proof of convergence of the IPFP in the multi-marginal case, follows a method similar to the one used in Theorem \ref{thm:convIPFP}.
\begin{teo}\label{thm:convIPFPNmarg}
Let $(X_1,d_1), \dots, (X_N,d_N)$ be Polish spaces, $\rho_1,\dots,\rho_N$ be probability measures in $X_1,\dots,X_N$, $c:X_1\times\dots\times X_N\to[0,+\infty]$ be a bounded cost, $p$ be an integer $1\leq p <\infty$. If $(a^n_j)_{n\in\N}, j\in\lbrace 1,\dots,N\rbrace$ are the IPFP sequence defined in \eqref{eq:IPFPsequenceN}, then there exist a sequence $\lambda^n \in \R^N$, with $\lambda^n_i >0$ and $\prod_{i=1}^N \lambda^n_i = 1$ such that
\[
\forall j\in\lbrace 1,\dots,N\rbrace, \quad  a^n_j/\lambda^n_j\to a_j \text{ in } L^p(\rho_j),
\]
where $(a_j)_{j=1}^N$ solve the Schr\"{o}dinger system. In particular, the sequence $\gamma^n = \Pi^N_{i=1}a^n_i\ka$ converges  in $L^p(\rho_1\otimes\dots\otimes\rho_N)$ to the optimizer $\gammaep_{opt}$  in \eqref{eq:primalSchrMult}. 
\end{teo}

\begin{proof} Let $i \in \lbrace 1,\dots, N\rbrace$ and consider Let $(a_i^n)_{n\in\N}$ the IPFP sequence defined in \eqref{eq:IPFPsequenceN}. For every $i$, we define $u_i^0 := \ep \ln (a_i^0)$ and then iteratively define the following potentials for every $p \in \N$
\[
\begin{cases}
u_1^{pN+1} = (\widehat{u_1^{pN}})^{(c,\ep)} \\
u_2^{pN+1} = u_2^{pN} \\ 
u_3^{pN+1} = u_3^{pN} \\ 
\, \cdots \\
u_N^{pN+1} = u_N^{pN} \\ 
     \end{cases}, \quad \begin{cases}
u^{pN+2}_1 =  u_1^{pN+1} \\
u_{pN+2} = (\widehat{u^{pN+1}_2})^{(c,\ep)} \\
u^{pN+2}_3 =  u_3^{pN+1} \\ 
\, \cdots  \\
u^{pN+2}_N =  u_N^{pN+1} \\ 
     \end{cases}, \dots \, , \quad \begin{cases}
u_1^{pN+N} = u_1^{pN+N-1} \\ 
u_2^{pN+N} = u_2^{pN+N-1} \\ 
u_3^{pN+N} = u_3^{pN+N-1} \\ 
\, \cdots \\
u^{pN+N}_N = (\widehat{u^{pN+N-1}_N})^{(c,\ep)} \\
     \end{cases}.
\] 	
Notice that $a_i^n = e^{u^{nN}_i / \ep }$ and moreover $osc ( u^{pN+i}_i  ) \leq \| c \|_{\infty}$ by Lemma \ref{lemma:entropytransbound} (iv), and in particular $osc(u^{n}_i) \leq \| c \|_{\infty}$ as long as $n \geq N$.  Moreover, thanks to \eqref{eqn:lambdau} we also have $|\sum_i \lambda_{u_i^n}| \leq \| c \|_{\infty}$. In particular, defining $v^n= P(u^n)$, we have $\|v_i^n\|_{\infty} \leq 3\|c\|_{\infty}$ thanks to Lemma \ref{lem:P} (ii); using \eqref{est:optcondmm} and Lemma \ref{lem:P} we also have
\[
\Dep^N(v^{n}_1,\dots, v^{n}_N)\leq \Dep^N(v^{n+1}_1,\dots, v^{n+1}_N) \leq \dots \leq \OTep(\rho_1,\rho_2,\dots, \rho_N).
\]

By the boundedness of $\| v_i^n\|_{\infty}$, by the compactness in Lemma \ref{lemma:entropytransbound} (iv), there exists a subsequence $k_n$  such that $v_i^{k_n}$ converges in $L^{p}$ to some $v_i$ for every $i=1, \ldots, N$; by pigeon-hole principle we have that at least a class of residue modulo $N$ is taken infinitely by the sequence $k_n$ and we will suppose that without loss of generality this residue class is $0$. Up to restricting to the infinite subsequence such that $k_n \equiv 0 \pmod{N}$, we can assume that $v_N^{k_{n}} = (\widehat{v_N^{k_n}})^{(N,c,\ep)}$ and $v_1^{k_{n}+1} = (\widehat{v_1^{k_n}})^{(N,c,\ep)}$

In particular, by the continuity of the $(N,c, \ep)$-transform we have
\[
\Dep^N(\hat{v_1}^{(N,c,\ep)}, \hat{v_1}) - \Dep^N(v_1,v_2,\dots,v_N) = \lim_{n\to\infty} \Dep(v_1^{k_n+1},\dots, v_N^{k_n+1}) - \Dep(v_1^{k_n},\dots, v_N^{k_n}) = 0.
\]
In particular, we have $v_1 = \hat{v_1}^{(N,c,\ep)}$ by \eqref{est:optcond2mm} and in particular $u_i^{k_n+1} \to u_i$ for every $i=1, \ldots, N$. Now, doing a similar computation, for every $i=2,\dots,N$, we  can inductively prove that, for every $i$,
\[
\Dep( \hat{v_i}^{(N,c,\ep)}, \hat{v_i}) - \Dep(v_1,\dots, v_i, \dots, v_N) = \lim_{n \to\infty} \Dep(v_1^{k_n+i},\dots, v_N^{k_n+i}) - \Dep(v_1^{k_n+i-1},\dots, v_N^{k_n+i-1}) = 0.
\]   
Hence, $v_i = \hat{v_i}^{(N,c,\ep)}, \, \forall i \in \lbrace 1,\dots, N\rbrace$. The result follows by noticing that $(e^{v_1/\ep},\dots, e^{v_N/\ep})$ solves the Schr\"odinger system, by Proposition \ref{prop:equiv_comp_multi}.

\end{proof}

\noindent
\emph{Remark on the multi-marginal problem $\Sep(\rho_1,\dots,\rho_N;\refm_1,\dots, \refm_N)$:} More generally, we could also consider the multi-marginal Schr\"odinger problem with references measures $\refm_i \in \P(X_i), i=1,\dots, N$. For simplicity, we denote $\overline{\rho} = (\rho_1,\dots,\rho_N)$ and $\overline{\refm} = (\refm_1,\dots,\refm_N)$. Then the functional $\Sep(\overline{\rho};\overline{\refm})$ is defined by

$$ \Sep(\overline{\rho};\overline{\refm}) = \min_{\gamma\in\Pi_N(\rho_1,\dots,\rho_N)} \ep\KLN(\gamma | \refm_1\otimes \cdots \otimes \refm_N).  $$

Analogously to the $2$ marginal case, the duality results, existence and regularity of entropic-potentials as well as the convergence of the Sinkhorn algorithm can be extended to that case. We omit the details here since the proof follows by similar arguments.\\

\noindent
{\bf Acknowledgements:} This work started when the second author visited the first author, while he was working at the Scuola Normale Superiore di Pisa (INdAM unit). The authors wants to thank G. Carlier, C. L\'eonard and L. Tamanini for useful discussions.

\section{Appendix}

\begin{prop}\label{prop:compactness} Let $(X,d,\mu)$ be a measurable metric space with $\mu(X)=1$. Let us assume that $\Fam \subset L^p(X,\mu)$ is a family of functions such that:
\begin{itemize}
\item[(a)] there exists $M>0$ such that $ \|f\|_{\infty} \leq M$ for every $f \in \Fam$;
\item[(b)] for every $\sigma$ there exists a set $N^{\sigma}$, a modulus of continuity $\omega_{\sigma}$ and a number $\beta_{\sigma}\geq0$ such that
$$ |f(x)-f(x')| \leq \omega_{\sigma}(d(x,x')) + \beta_\sigma \qquad \forall x,x' \not \in N^{\sigma}$$
where $N^{\sigma}$ and $\beta_{\sigma}$ are such that $\mu(N^{\sigma})+\beta_\sigma \to 0$ as $\sigma \to 0$.
\end{itemize}
Then the family $\Fam$ is precompact in $L^p(X, \mu)$.
\end{prop}

\begin{proof}
Let us fix $\ep>0$ and let us consider a sequence $\sigma_n \to 0$ such that $\sum_{n =1}^{\infty} \mu(N^{\sigma_n}) \leq \ep$; then define $\omega_n= \omega_{\sigma_n}$ and  $\mathcal{N}^{\ep} := \bigcup_n N^{\sigma_n}$; in particular we  have $\mu(\mathcal{N}^{\ep}) \leq \ep$ and
\begin{equation}\label{eqn:eqcont} |f(x)-f(x')| \leq\omega_n(d(x,x'))  + \beta_{\sigma_n} \qquad \forall x, x' \not \in \mathcal{N}^{\ep}, \forall n \in \N.\end{equation}
Let us define $\omega^{\ep}(t) = \inf_n \{ \omega_n (t) + \beta_{\sigma_n}\}$: by \eqref{eqn:eqcont} we have that $f$ is $\omega^{\ep}$-continuous outside $\mathcal{N}^{\ep}$.

We can verify that $\omega^{\ep}$ is a non degenerate modulus of continuity: it is obvious that is it nondecreasing since it is an infimum of noncreasing functions. Then for every $\tilde{\ep}>0$ we can choose $n$ big enough such that $\beta_{\sigma_n}< \tilde{\ep}/2$ and then choose $t$ small enough such that $\omega_n (t) < \tilde{\ep}/2$; in this way we have $\omega^{\ep}(t) \leq   \omega_n (t) + \beta_{\sigma_n} < \tilde{\ep}$. In particular $\omega^{\ep}(t) \to 0$ as $t \to 0$.

Now we conclude by a diagonal argument: let us consider a sequence $(f^0_n)_{n \in \N} \subseteq \Fam$ and a sequence $\ep_k\to 0$. We want to find a subsequence that is converging strongly in $L^p$. We iteratively extract a subsequence $(f^k_n)$ of $(f^{k-1}_n)$ that is converging uniformly outside $\mathcal{N}^{\ep_k}$ (thanks to Ascoli-Arzel\`a) to some function $f^k$, which is defined only outside $\mathcal{N}^{\ep_k}$. Then let us consider
$$ f(x)= \begin{cases} f^k(x) \qquad &\text{ if }x \not \in \mathcal{N}^{\ep_k} \\ 0 & \text{ otherwise.} \end{cases}$$
First of all $f$ is well defined since if $x \not \in \mathcal{N}^{\ep_k}$ and $x \not \in \mathcal{N}^{\ep_j}$ with $j >k$ then we have that $f^k_n (x) \to f^k(x)$ but since $f^j_n$ is a subsequence of $f^k_n$ we have also $f^j_n(x) \to f^k(x)$; however by definition  $f^j_n(x) \to f^j(x)$ and so $f^j(x)=f^k(x)$. Moreover it is clear that $\|f\|_{\infty} \leq M$ since this is true for every $f^0_n$ thanks to property (a). Now we consider the sequence $g_n = f^n_n$ which is a subsequence of $f^0_n$. Let us fix $\ep>0$ and choose $k$ such that $\ep_k < \ep^p$; then let $n_0>k$ such that $|f^k_n - f| \leq \ep$ on $X \setminus \mathcal{N}^{\ep_k}$ for every $n \geq n_0$. Now we have $g_{n}=f^k_n$ for some $n\geq n_0$ and in particular

\begin{align*} \int_X |g_{n_0}(x)-f(x)|^p \,d \mu & = \int_{X} | f^k_n(x) -f(x)|^p \, d \mu  \\
&= \int_{\mathcal{N}^{\ep_k}} | f^k_n(x)-f(x)|^p \, d \mu + \int_{X \setminus  \mathcal{N}^{\ep_k}} | f^k_n(x)-f(x)|^p \, d \mu  \\
& \leq \int_{\mathcal{N}^{\ep_k}} (2M)^p \, d \mu + \int_{X \setminus  \mathcal{N}^{\ep_k}} \ep^p \, d \mu  \\
& \leq \mu(\mathcal{N}^{\ep_k}) (2M)^p + \ep^p ) \\
& \leq \ep_k (2M)^p + \ep^p \mu(X) \leq \ep^p ( 2^pM^p+ 1)
\end{align*}

In particular we get $g_n \to f $ in $L^p$ and so we're done.

\end{proof}

\bibliographystyle{siam}
\bibliography{refs}

\end{document}